\documentclass[12pt]{amsart}

\usepackage{fancyhdr}
\usepackage{amsthm, amscd, mathdots, longtable, young}
\usepackage{makecell}

\usepackage{setspace,amsthm,amsmath,amssymb,url}
\usepackage[mathscr]{euscript}
\usepackage{geometry}        
\usepackage{graphicx}
\usepackage{amssymb}
\usepackage[table]{xcolor}
\usepackage{amsfonts}
\usepackage{mathrsfs}
\usepackage{upgreek}

\usepackage{float}
\usepackage{bm}
\usepackage[mathscr]{euscript}
\usepackage[all]{xy}
\usepackage{epsfig}
\usepackage[T1]{fontenc}

\usepackage{amsmath}
\usepackage{amsthm}
\usepackage{mathrsfs}
\usepackage{epstopdf}
\usepackage{url}
\usepackage{hyperref}
\usepackage[msc-links,alphabetic]{amsrefs}
\usepackage{upgreek}

\usepackage{subfig} 

\newcommand{\ep}{\varepsilon}

\renewcommand{\phi}{\varphi}
\newcommand{\si}{\sigma}

\newcommand{\ZZ}{{\mathbb Z}}

\newcommand{\QQ}{{\mathbb Q}}
\newcommand{\RR}{{\mathbb R}}

\newcommand{\sig}{\operatorname{sig}}

\newcommand{\oo}{{\bf 1}}
\newcommand*\wbar[1]{
  \hbox{ \kern-0.2em%
    \vbox{%
      \hrule height 0.5pt  
      \kern0.25ex
      \hbox{%
        \kern-0.15em
        \ensuremath{#1}%
        \kern-0.05em
      }%
    }%
  \kern0.05em}%
}

\newtheorem{theorem}{Theorem}[section]
\newtheorem{lemma}[theorem]{Lemma}
\newtheorem{proposition}[theorem]{Proposition}

\newtheorem{corollary}[theorem]{Corollary}
\newtheorem*{main}{Theorem}

\theoremstyle{definition}     
\newtheorem{definition}[theorem]{Definition}

\theoremstyle{remark}
\newtheorem*{remark}{Remark}
\newtheorem{example}{Example}

\title[The Khovanov homology of alternating virtual links]{The Khovanov homology of alternating virtual links}

\author[H. Karimi]{Homayun Karimi}
\address{Mathematics \& Statistics, McMaster University, Hamilton, Ontario}
\email{karimih@math.mcmaster.ca}


\date{\today}                 

\begin{document}

\begin{abstract}
In this paper, we study the Khovanov homology of an alternating virtual link $L$ and show that it is supported on $g+2$ diagonal lines, where $g$ equals the virtual genus of $L$. Specifically, we show that $Kh^{i,j}(L)$  is supported on the lines $j=2i-\sigma_{\xi}+2k-1$ for  $0\leq k\leq g+1$ where $\si_{\xi^*}(L)+2g= \si_{\xi}(L)$ are the signatures of $L$ for a checkerboard coloring $\xi$ and its dual $\xi^*$. Of course, for classical links, the two signatures are equal and this recovers Lee's $H$-thinness result for $Kh^{*,*}(L)$. 
Our result applies more generally to give an upper bound for the homological width of the Khovanov homology of any checkerboard virtual link $L$. The bound is given in terms of the \emph{alternating genus} of $L$, which can be viewed as the virtual analogue of the Turaev genus. The proof rests on associating, to any checkerboard colorable link $L$, an alternating virtual link diagram with the same Khovanov homology as $L$. 

In the process, we study the behavior of the signature invariants under vertical and horizontal mirror symmetry. We also compute the Khovanov homology and Rasmussen invariants in numerous cases and apply them to show non-sliceness and determine the slice genus for several virtual knots. Table \ref{ras-table} at the end of the paper lists the signatures, Khovanov polynomial, and Rasmussen invariant for alternating virtual knots up to six crossings.
\end{abstract}

\maketitle
\section*{Introduction}

Khovanov homology was
introduced in \cite{Khovanov-00}.
It is a powerful invariant that is known to detect the unknot by deep results of Kronheimer and Mrowka (\cite{KM-2011}). In \cite{Lee}, Lee modified the differential to define a new homology. The Lee homology of a knot is equivalent to an even integer, which surprisingly yields a powerful concordance invariant, as shown by Rasmussen in \cite{Rasmussen}. Lee proved that the Khovanov homology of alternating links is supported in two lines.  Her result implies that for an alternating knot, Rasmussen's invariant is equal to minus the signature of that knot. She also proved that for alternating knots, Khovanov homology is determined by the signature and the Jones polynomial. To see a generalization of Lee's theorem to tangles see \cite{Bar14}.

Manturov extended Khovanov homology to virtual knots and links, first with $\ZZ/{2}$ coefficients \cite{Manturov-kh} and later for arbitrary coefficients \cite{Manturov-2007}. In \cite{DKK}, Dye, Kaestner and Kauffman reformulated Manturov's approach and extended Lee homology theory to the virtual setting. They also used it to define a Rasmussen invariant for virtual knots. As we shall see, for virtual knots and links, Khovanov homology is not as powerful an invariant as it is for classical knots. For instance, there exist nontrivial virtual knots with trivial Khovanov homology (see Example \ref{ex37}).

Signatures were extended to checkerboard colorable virtual links in \cite{Im-Lee-Lee-2010}, and they depend not only on the link $L$ but also on the choice of checkerboard coloring $\xi$. In particular, instead of a single signature, we have a pair of signatures $(\sigma_\xi,\sigma_{\xi^*})$, where $\xi^*$ is the dual coloring. In Theorem \ref{mirror images}, we examine how the signatures $(\sigma_\xi,\sigma_{\xi^*})$ change under taking the vertical and horizontal mirror images. If $D$ is an alternating link diagram for $L$ with supporting genus equal to $g$, then by Theorem 5.19 in \cite{Karimi}, we have $\sigma_\xi-\sigma_{\xi^*}=2g$. 
In this paper, we prove  that the Khovanov homology of $L$ is supported in $g + 2$ lines, where $g$ is the virtual genus of $L$.

\begin{main}\label{theorem1}
If $D$ is a connected alternating virtual link diagram with genus $g$, and signatures $\sigma_{\xi},\sigma_{\xi^*}$, then its Khovanov homology is supported in $g+2$ lines: $$j=2i-\sigma_{\xi^*}+1,j=2i-\sigma_{\xi^*}-1,\ldots,j=2i-\sigma_{\xi}-1.$$
\end{main} 

This theorem is the analogue for virtual links of Lee's result on $H$-thinness of Khovanov homology for an alternating classical links \cite{Lee}. It is proved in Section \ref{four} (cf.~Theorem \ref{prop}) by induction on 
the number of crossings of $D$.

For each non-split link diagram, one can construct a surface called the Turaev surface, such that the diagram is alternating on that surface. Given a non-split link $L$, the minimum genus over all diagrams and all surfaces is denoted $g_T(L)$ and called the \emph{Turaev genus}. The Turaev genus equals zero if and only if the link is alternating. In other words, the Turaev genus measures how far a given link is from being alternating.

For classical links, the Turaev genus provides an upper bound for the homological width of the Khovanov homology \cite{Kofman07}. For a given checkerboard diagram $D$ of a virtual link $K$, we can associate an alternating diagram, $D_{alt}$, to $D$, without changing the Khovanov homology. Suppose the new diagram $D_{alt}$ has supporting genus $g$, then by the previous theorem,
 its Khovanov homology, which is the same as the Khovanov homology of $D$, is supported in $g+2$ lines. 
 
For a checkerboard colorable virtual link $L$, we define the alternating genus $g_{alt}(L)$ to be the minimum, over all checkerboard diagrams $D$ for $L$, of the supporting genus of $D_{alt}$. Corollary \ref{cor3} shows that the alternating genus provides an upper bound for the homological width of the Khovanov homology.
When $L$ is classical and non-split, Lemma \ref{lem-alt-Turaev} implies that
$g_{alt}(L) \leq g_T(L)$; thus our result recovers and generalizes the Turaev genus bound for the homological width
of Khovanov homology of classical links (cf.~Corollary 3.1 of \cite{Kofman07}).
 
We also give new computations of the Rasmussen 
invariant, and we apply it to show non-sliceness of several virtual knots.

In Section \ref{sec-1}, we introduce the basic notions of virtual knot theory. In Section \ref{sec-sign}, we define checkerboard colorability and the signatures for virtual links. In Section \ref{two}, we recall the definition of the Khovanov homology for classical knots and links. In Section \ref{three}, we review the extension of Khovanov homology to virtual knots and links. The main result is proved in Section \ref{four}, and in Section \ref{five}, we present computations of signatures, Khovanov polynomials, and Rasmussen invariants of alternating virtual knots up to six crossings.  

\section{Basic Notions} \label{sec-1}

In this section we recall some basic definitions from virtual knot theory, including virtual link diagrams, abstract link diagrams, supporting genus, virtual genus, virtual knot concordance, and the slice genus. 
We then introduce the notion of checkerboard coloring for virtual knots and links. We also recall the notion of the boundary property for virtual link diagrams (see \cite{Boundary}), and relate it to checkerboard colorability.

\subsection*{Virtual link diagrams} A \emph{virtual link diagram} 
is a collection of immersed closed curves in the plane, with a finite number of intersection points all of which are double points. Each double point is either classical or virtual. At classical crossings, we record extra information by specifying which of the two strands goes over the other, and at virtual crossings we  place a small circle around the double point. A virtual link is an equivalence class of virtual link diagrams modulo the \emph{generalized Reidemeister moves} and planar isotopy. The combination of classical and virtual Reidemeister moves is called the generalized Reidemeister moves. See Figure \ref{grms}.

\vspace{.5cm}
\begin{figure}[h!]
\centering\includegraphics[height=20mm]{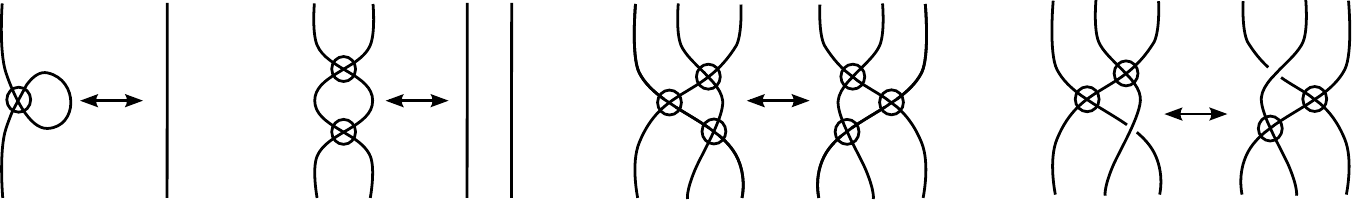}
\caption{The virtual Reidemeister moves.}\label{grms}
\end{figure}
\vspace{.3cm}

\subsection*{Abstract link diagrams} Suppose $S$ is a compact oriented surface with boundary. Let $D$ be a link diagram on $S$ with no virtual crossings. We denote by $|D|$, the graph obtained by replacing each classical crossing in $D$ by a tetravalent vertex. We say $P=(S,D)$ is an \emph{abstract link diagram}\index{abstract link diagram} (ALD) if $|D|$ is a deformation retract of $S$.   

Let $\Sigma$ be a closed, connected and oriented surface and $f:S\to \Sigma$ be an orientation preserving embedding. We call $(\Sigma,f(D))$ a \emph{realization}\index{realization} of $P$. 

Given a virtual link diagram, we can construct an abstract link diagram (see \cite{KK00}). We review that construction here. 

Let $D$ be a virtual link diagram with $n$ classical crossings and
$U_{1},U_{2},\ldots,U_{n}$ regular neighborhoods of the crossings of $D$. Put $W=\text{cl}(\RR^2-\cup_{i=1}^{n}U_{i})$. Thickening the arcs and loops of $D\cap W$,
we obtain immersed bands and annuli in $W$ whose cores are $D\cap W$. Their union  together with $U_{1},U_{2},\ldots,U_{n}$ forms an immersed disk-band surface $N(D)$ in
the plane. Modifying $N(D)$ as shown below at each virtual crossing of $D$, we obtain a
compact oriented surface $S_{D}$ embedded in $\RR^3$, and a diagram $\widetilde{D}$ on $S_{D}$ corresponding to $D$. We call the pair $P=(S_{D},\widetilde{D})$ the \emph{abstract link diagram associated to $D$}. 

\vspace{.3cm} 
\begin{figure}[h!]
\centering\includegraphics[height=20mm]{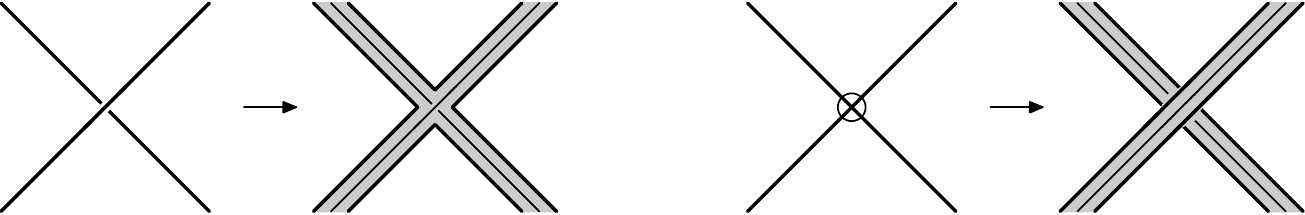}
\caption{Modifying $N(D)$ at a virtual crossing.}
\end{figure}
\vspace{.3cm}

The \emph{supporting genus}\index{supporting genus} of an ALD $P=(\Sigma,\widetilde{D})$ is denoted by $sg(P)$, and is defined to be the minimal genus among the realization surfaces $F$ of $P$. The supporting
genus of a virtual link diagram $D$ is defined to be the supporting genus of
the ALD $P=(S_{D},\widetilde{D})$ associated with $D$ and denoted by $sg(D)$. The \emph{virtual
genus}\index{virtual
genus} of a virtual link $L$ is denoted by $g_v(L)$ and defined to be the minimum number among the
supporting genus $sg(D)$, where $D$ runs over all virtual link diagrams representing
$L$.

Let $L$ be a virtual link. A virtual link diagram $D$ representing $L$ such that $sg(D)=g_v(L)$ is called a \emph{minimal diagram}\index{minimal diagram} of $L$.

\subsection*{Virtual knot concordance} We recall the notions of cobordism and concordance of virtual knots.
\begin{definition}
\noindent {\bf (i)} Two knots $K_0\subset \Sigma_0\times I$ and $K_1\subset \Sigma_1\times I$ in thickened surfaces are called \emph{virtually cobordant}\index{virtually cobordant} if there exists a compact connected oriented
$3$-manifold $W$ with $\partial W\cong-\Sigma_0\sqcup\Sigma_1$ and an oriented surface $S \subset W\times I$ with $\partial S = -K_0\sqcup K_1$.\\
\noindent {\bf (ii)} The knots $K_0, K_1$ in part (i) are called \emph{virtually concordant} if the surface $S$ can be chosen to be an annulus.
\end{definition} 

Given a knot $K$ in the thickened surface $\Sigma\times I$, an elementary argument shows that there exists a compact oriented $3$-manifold $W$ and compact oriented surface $S \subset W\times I$ with $\partial S =K.$ Consequently, every virtual knot is cobordant to the unknot.
This motivates the following definition.  

\begin{definition}
Suppose $K$ is a knot in a thickened surface $\Sigma \times I.$ \\
\noindent {\bf (i)} The \emph{slice genus}\index{slice genus} of $K$, denoted $g_{s}(K)$, is the minimum genus of $S$, over all  3-manifolds $W$ with $\partial W =\Sigma$ and over all surfaces $S \subset W \times I$ with $\partial S =K.$ \\
\noindent {\bf (ii)} The knot $K$ is called \emph{virtually slice}\index{virtually slice} if $g_s(K)=0.$ Equivalently, $K$ is virtually slice if it bounds a disk $\Delta\subset W\times I$.
\end{definition}

\section{Signatures of checkerboard colorable virtual links} \label{sec-sign}

\subsection*{Checkerboard colorings} We review checkerboard colorings for virtual knots and links and recall the construction of the Goeritz matrices, which are used to define signatures for checkerboard colorable virtual links.
 
\begin{definition}
Given $P=(F,D)$, where $F$ is a compact, connected, oriented surface and $D$ is a link diagram on $F,$ a \emph{checkerboard
coloring} $\xi$ is an assignment to each region of $F \smallsetminus |D|$ one of two colors, say black and white, such that any two
adjacent regions sharing an edge of $|D|$ have different colors. Define the dual checkerboard coloring $\xi^*$ to be the one obtained from $\xi$ by interchanging
black and white regions.
\end{definition}

For a crossing $c$, we have the following pictures:

\vspace{.3cm} 
\begin{figure}[h!]
\centering\includegraphics[height=24mm]{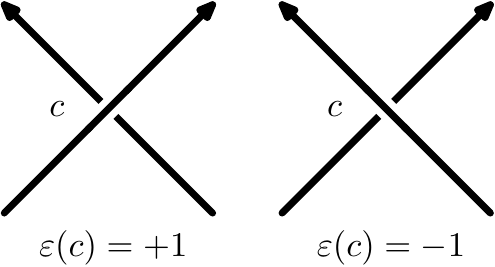} \qquad \includegraphics[height=24mm]{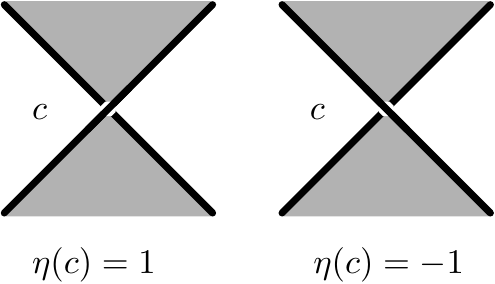}  \qquad \includegraphics[height=24mm]{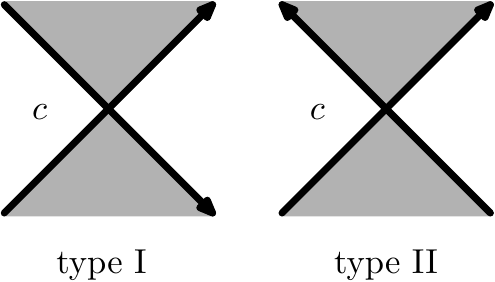}
\caption{From left to right, a positive and negative crossing, a type $A$ and type $B$ crossing, and a type I and type II crossing.}
\end{figure}
\vspace{.3cm}

They are a positive crossing and a negative crossing, a type $A$ crossing with $\eta(c)=+1$ and a type $B$ crossing with $\eta(c)=-1$ in a checkerboard colored link diagram, and a type I and a type II crossing in an oriented, colored link diagram, respectively. We call $\ep(c)$ the \emph{writhe} of the crossing and $\eta(c)$ the \emph{incidence} of the crossing $c$. An elementary calculation shows that a crossing $c$ has type II if $\ep(c) \eta(c) = 1,$ otherwise it has type I.

Let $\xi$ be a checkerboard coloring for a pair $P=(F,D)$, where $F$ is a closed, oriented and connected
surface and $D$ is a link diagram on $F$. 
We enumerate the white regions of
$F \smallsetminus |D|$ by $X_{0},X_{1},\ldots,X_{m}$. Let $C(D)$ denote the set of all classical crossings of
$D$ on $F$. For each pair $i,j\in\{1,2,\ldots,m\}$, 
let $$C_{ij}(D)=\{c\in C(D)\ \mid\ c\ \text{is adjacent to both}\ X_{i}\ \text{and}\ X_{j}\},$$ and
define  
$$
g_{ij}=\begin{cases}
-\sum\limits_{c\in C_{ij}(D)}\eta(c),\ \ \text{for}\ i\neq j,\\
-\sum\limits_{k=0; k\neq i}^m g_{ik},\ \ \text{for}\ i=j.
\end{cases}
$$

The pre-Goeritz matrix of $D$ is defined to be the symmetric integral
matrix $G_\xi'(D)=(g_{ij})_{0\leq i,j\leq m}$, and the Goeritz matrix of $D$ is the principal minor $G_\xi(D)=(g_{ij})_{1\leq i,j\leq m}$ obtained by removing the first row and column of $G_\xi'(D).$

The correction term is defined by setting $\mu_{\xi}(D)=\sum\limits_{c\ \text{is type II}}\eta(c)$. 

\begin{definition} Suppose $D$ is a checkerboard colorable link diagram on a surface $F$ with a checkerboard coloring $\xi$, we define the signature as follows:
$$\sigma_{\xi}(D)=\sig(G_\xi(D))-\mu_{\xi}(D).$$
\end{definition}

The following result is proved in \cite{Im-Lee-Lee-2010}.
\begin{theorem}[Im-Lee-Lee]
If $L$ is a non-split checkerboard link represented by a diagram $D$ of minimal genus and with coloring $\xi$, then the pair $\{\sigma_{\xi}(D), \sigma_{\xi^*}(D)\}$
of signatures is independent of the choice of virtual link diagram and gives a well-defined invariant of the virtual link $L$. 
\end{theorem}

Given a virtual link diagram $D$, for each classical crossing, we can resolve the crossing into a $0$-smoothing or a $1$-smoothing (see Figure \ref{resolving}).  

\vspace{.3cm} 
\begin{figure}[h!]
\centering\includegraphics[height=25mm]{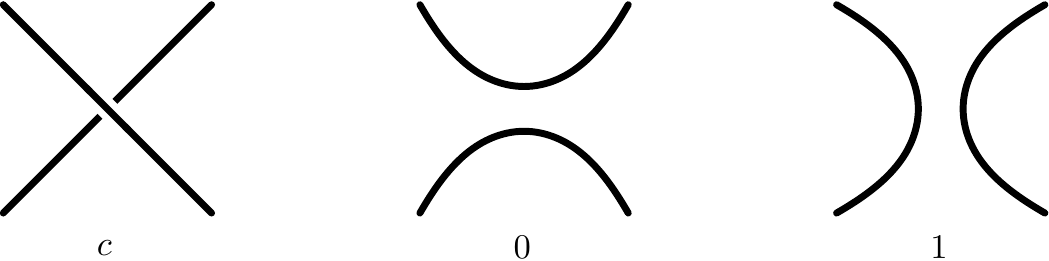}
\caption{The $0$- and $1$-smoothing of a crossing.}\label{resolving}
\end{figure}
\vspace{.3cm}

If we resolve all the classical crossings, the resulting diagram is called a \emph{state}\index{state}. A state is a virtual link diagram with only virtual crossings, i.e.~it is a diagram of the unlink. For a link diagram with $n$ classical crossings, we have $2^{n}$ states.  In fact, once an ordering of the crossings $\{c_1, \ldots,c_n\}$ has been fixed, the states are in one-to-one correspondence with binary strings of length $n$. For a given state $s$, the dual state is denoted $\wbar{s}$ and it is  obtained from $s$ by changing all 0-smoothings to 1-smoothings, and vice versa. In other words, if $s$ corresponds to the binary word with $i$-th entry $s_i \in \{0,1\}$, then $\wbar{s}$ corresponds to the binary word with $i$-th entry $\wbar{s}_i = 1-s_i.$

\begin{definition}\label{boundary}
Let $D$ be a virtual link diagram, and $(S_{D},\widetilde{D})$ be the abstract link diagram associated with $D$. Then $D$ has the \emph{boundary property}\index{boundary property} if there exists a state $s_{\partial}$ such that $\partial S_{D}=s_{\partial}\cup \wbar{s}_{\partial}$, where $\wbar{s}_{\partial}$ is the dual state of $s_{\partial}$. 
\end{definition} 

The following lemma relates the boundary property to checkerboard colorability. 

\begin{lemma} \label{lem-bound-prop}
A virtual link diagram $D$ has the boundary property if and only if it is checkerboard colorable. 
\end{lemma} 

\begin{proof}
Suppose $D$ has the boundary property and define a checkerboard coloring $\xi$ as follows. Let the white regions be those with boundary a component of $s_{\partial}$, and let the black regions be those with boundary a components of $\wbar{s}_{\partial}$. This gives a checkerboard coloring $\xi$ for $D$.

Conversely, suppose $\xi$ is a checkerboard coloring of the abstract link diagram $(S_D, \widetilde{D}).$ Let $s_\partial$ be the state obtained by performing 0-smoothing to all crossings $c$ with $\eta(c) = +1$ and 1-smoothing to all crossings $c$ with $\eta(c) = -1$, and let $\wbar{s}_{\partial}$ be the dual state. Then it can be easily checked that $\partial S_D = s_\partial \cup \wbar{s}_\partial,$ therefore $D$ has the boundary property.
\end{proof}

\subsection*{Alternating virtual links}
A virtual link is called \emph{alternating} if it admits a virtual link diagram whose crossings alternate between over and under crossing as one travels around each component of the link.
Since every alternating virtual link diagram is checkerboard colorable (see \cite{Kamada-2002}), it follows that every alternating virtual link diagram has the boundary property. 

Let $|s_{\partial}|$ and $|\wbar{s}_{\partial}|$ be the number of components of $s_{\partial}$ and $\wbar{s}_{\partial}$, respectively. 

\begin{lemma}\label{boundary lemma}
Suppose $D$ is a virtual link diagram with $n$ classical crossings. 
If $S_{D}$ has genus $g$ and $D$ has the boundary property, then  
$$|s_{\partial}|+|\wbar{s}_{\partial}|=n+2-2g.$$ 
\end{lemma}

\begin{proof}
Attach disks to the boundary components of $S_{D}$ to get a closed surface $\Sigma$. There is a cell decomposition on $\Sigma$, defined as follows: There is a one-to-one correspondence between  the classical crossings of $D$ and $0$-cells, bands of $S_{D}$ and $1$-cells, and $2$-disks that we attached to $S_{D}$ and $2$-cells. The Euler characteristic of $\Sigma$ is $2-2g$. And the number of $0$, $1$ and $2$-cells are $n$, $2n$ and $|s_{\partial}|+|\wbar{s}_{\partial}|$, respectively. The lemma now follows.      
\end{proof}

The proof of the next result is elementary and is left to the reader.

\begin{lemma}\label{eta}
If $D$ is a connected checkerboard colorable link diagram, then $D$ is alternating if and only if all its crossing have the same incidence number.
\end{lemma}

%

\subsection*{Unknotting operations}

For a virtual link, we introduce the operations of  \emph{orientation reversal} ({\bf or}), \emph{sign change} ({\bf sc}), and \emph{crossing change} ({\bf cc}) at a given crossing. 
The operations {\bf or} and {\bf sc} are shown in Figure \ref{fig3}, and {\bf cc} is the result of applying  
{\bf or} and {\bf sc}.
As is well-known, crossing change {\bf cc} is an unknotting operation for classical knots and links, but this is no longer true for virtual links. Together with the operation of \emph{chord deletion} ({\bf cd}), which replaces a classical crossing with a virtual one, these moves form a complete set of unknotting operations for virtual knots (see \cite{Boden2}). 
Note that each of {\bf cc}, {\bf or} and {\bf sc} can be achieved in a genus one cobordism.

\begin{figure}[h!]
\centering\includegraphics[height=15mm]{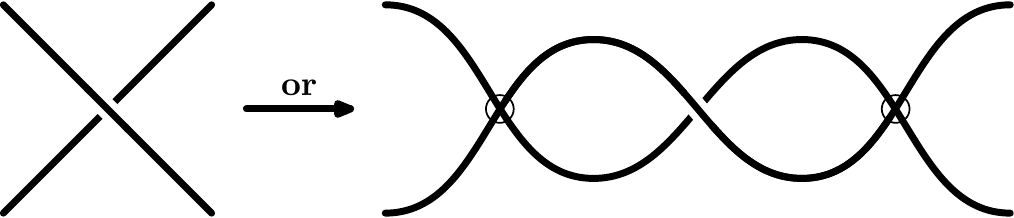}
\qquad
\centering\includegraphics[height=15mm]{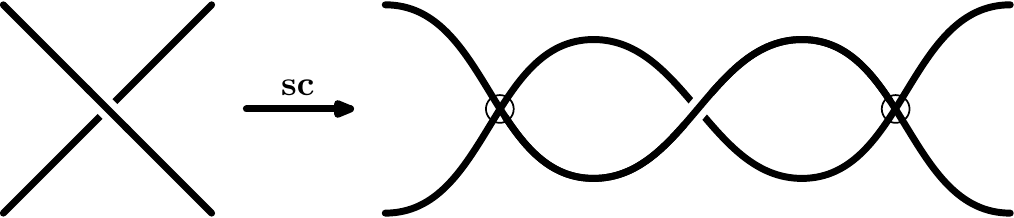}
\caption{The operations {\bf or} and {\bf sc}.}
\label{fig3}
\end{figure}


Given a checkerboard colored diagram, if we apply any one of $\{{\bf cc}, {\bf sc}, {\bf or}\}$ to a crossing,
then the new diagram is again checkerboard colored. We will examine the effect of these operations on the writhe, incidence, and type of the crossing.

If we apply {\bf or} to a crossing $c$, then it is elementary to see that $\eta(c)$ changes sign and the writhe $\ep(c)$ remains the same. If we instead apply {\bf sc}, then the writhe $\ep(c)$ changes sign and $\eta(c)$ remains the same. In particular, under either operation, the type of the crossing changes. 
On the other hand, if we apply {\bf cc} to a crossing $c$, then both the writhe $\ep(c)$ and $\eta(c)$  change signs, but the type remains the same.
These facts are summarized in Table \ref{effect}. 

\begin{table}[h!]
\begin{center}
{\rowcolors{1}{white!80!white!50}{lightgray!70!lightgray!40}
\begin{tabular}{|c|c|c|c|}  \hline
& writhe & incidence & type \\
\hline
 {\bf sc}&$ -\varepsilon(c)$ & $\eta(c)$ & $-\text{type}(c)$ \\
 {\bf or}& $\varepsilon(c)$ & $-\eta(c)$ & $-\text{type}(c)$ \\
 {\bf cc}& $-\varepsilon(c)$ & $-\eta(c)$ & $\text{type}(c)$ \\ \hline
\end{tabular}
}  
\end{center}
\caption{The effect of applying {\bf sc}, {\bf or}, {\bf cc} to the crossing $c$.}\label{effect}
\end{table}

\subsection*{Mirror images of virtual knots}
We define the vertical and horizontal mirror image of a virtual knot and relate their signatures.
\begin{definition}
Given an oriented virtual knot diagram $D$, let $-D$ be the diagram obtained by reversing the orientation of $D$.
The \emph{vertical mirror image} of $D$ is denoted $D^{*}$
and is the diagram obtained by applying {\bf cc} to all crossings of $D$.
The \emph{horizontal mirror image} of $D$ is denoted $D^{\dag}$ and is 
the diagram obtained by applying {\bf sc} to all crossings of $D$.
\end{definition}

\begin{lemma} If $D$ is a minimal genus diagram for a virtual knot $K$, then so are $-D, D^*$ and $D^\dag$.
\end{lemma}

\begin{proof}
 It is obvious that if $D$ is a minimal genus diagram for $K$, then $-D$ is a minimal genus diagram for $-K$.

Suppose that $P=(S_D,\widetilde{D})$ is the abstract link diagram associated with $D$.
Place it inside 
$\{(x,y,z)\in\RR^3\mid y<0\}$ in such a way that the projection of $\widetilde{D}$ on the $xy$-plane is $D$. Now reflect $P$ with respect to the plane $y=0$. The result is an abstract link diagram associated with
$D^{\dag}$. This shows that if $D$ is a minimal genus diagram, then $D^{\dag}$ is also minimal genus.

On the other hand, switching all the over-crossings and under-crossings in $\widetilde{D}$, 
we obtain an abstract link diagram for $D^{*}$. It 
follows that if $D$ is a minimal genus diagram, then $D^{*}$ is also minimal genus. 
\end{proof}

Suppose $\xi$ is a checkerboard coloring of $D$ and $\xi^*$ is its dual coloring. Notice that a coloring is determined by the
underlying flat knot. Therefore we can use the same notation for the colorings of the diagrams of the 
mirror images and the inverse knot.  

The following picture shows  a colored crossing in $D$ (left) and $-D$ (right).

\vspace{.3cm}
\begin{figure}[h!]
\centering\includegraphics[height=25mm]{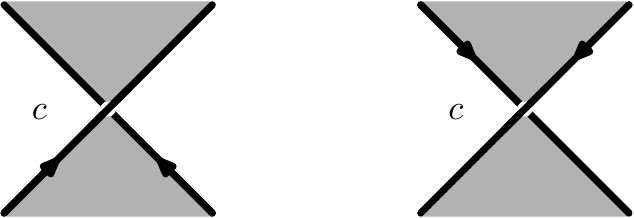}
\caption{A colored crossing in $D$ and $-D$.}
\end{figure}
\vspace{.3cm}

Therefore, at each crossing, the incidence number and type of that crossing are unchanged. Notice that black and white regions are also unchanged. Thus the two signatures for $D$ and $-D$ are the same. 

For $D^{*}$,  at each crossing, the type is unchanged but the incidence 
number changes sign. As a result, both the Goeritz matrix and the correction term are 
multiplied by $-1$. Thus $\sigma_{\xi}(D^{*})=-\sigma_{\xi}(D)$ and
$\sigma_{\xi^*}(D^{*})=-\sigma_{\xi^*}(D)$. 

For $D^{\dag}$, we use the dual coloring $\xi^*$. Thus at each crossing, the type is unchanged but the incidence 
number changes by a negative sign. Therefore $\sigma_{\xi^*}(D^{\dag})=-\sigma_{\xi}(D)$, and $\sigma_{\xi}(D^{\dag})=-\sigma_{\xi^*}(D)$. 

Since ${\bf or} = {\bf sc} \circ {\bf cc},$ it follows that
$D^{*\dag}$ is obtained by applying {\bf or} to all crossings of $D$.
For $D^{*\dag}$, we find that
$\sigma_{\xi^*}(D^{*\dag})=\sigma_{\xi}(D)$ and $\sigma_{\xi}(D^{*\dag})=\sigma_{\xi^*}(D)$.

The next theorem summarizes these observations. 

\begin{theorem}\label{mirror images}
If $K$ is a virtual knot with checkerboard coloring $\xi$, then the signatures of the inverse and  mirror images of $K$ satisfy:
\begin{eqnarray*}
(\sigma_{\xi}(-K),\sigma_{\xi^*}(-K))&=&(\sigma_{\xi}(K),\sigma_{\xi^*}(K)),\\
(\sigma_{\xi}(K^{*}),\sigma_{\xi^*}(K^{*}))&=&(-\sigma_{\xi}(K),-\sigma_{\xi^*}(K)),\\
(\sigma_{\xi}(K^{\dag}),\sigma_{\xi^*}(K^{\dag}))&=&(-\sigma_{\xi^*}(K),-\sigma_{\xi}(K)),\\
(\sigma_{\xi}(K^{*\dag}),\sigma_{\xi^*}(K^{*\dag}))&=&(\sigma_{\xi^*}(K),\sigma_{\xi}(K)).
\end{eqnarray*}
\end{theorem}



\section{Khovanov Homology for Classical Knots}\label{two}

In this section, we briefly introduce the Khovanov homology for classical knots and links. For more details see \cite{Turner} and \cite{Bar-natan-02}.

Khovanov homology is a $(1+1)$-TQFT (\emph{topological quantum field theory}\index{topological quantum field theory}), i.e.~it is a functor from the category of compact $1$-dimensional manifolds (a collection of circles) with morphisms, compact and  orientable $2$-dimensional cobordisms (surfaces) between them, into the category of graded vector spaces and graded linear maps.

Khovanov introduced the invariant for classical links in \cite{Khovanov-00}. It is a bigraded homology theory which is defined and computed in a purely combinatorial way. Khovanov homology is a categorification of the Jones polynomial, in that its graded Euler characteristic is equal to the unnormalized Jones polynomial. For a link $L$, we denote its Khovanov homology by $Kh^{*,*}(L)$, and we have
$$\widehat{\chi}(Kh^{*,*}(D))=\sum_{i,j\in \ZZ}(-1)^iq^j\text{dim}Kh^{i,j}(D)=\widehat{V}_L(q).$$

Suppose $D$ is a link diagram with $n_{+}$ positive crossings and $n_{-}$ negative crossings. Let $n=n_{+}+n_{-}$ and enumerate the crossings by $c_1,\ldots,c_n$. With $\QQ$ as the coefficient ring, we set $V=\QQ \oo\oplus\QQ X$ to be the $2$ dimensional vector space with basis $\{\oo,X\}$. Setting the degree of $\oo$ to be $+1$ and the degree of $X$ to be $-1$ gives $V^{\otimes n}$ the structure of a graded vector space. This grading will be denoted $j$ and called \emph{vertical}\index{vertical grading} or \emph{quantum grading}\index{quantum grading}. 

If $W=\bigoplus_{m\in \ZZ}W_{m}$ is a graded vector space, then a \emph{vertical grading shift}\index{vertical grading shift} of $W$ by $\ell$ is defined as $W\{\ell\}=\bigoplus_{m\in\ZZ}W_{m}'$, where $W_{m}'=W_{m-\ell}$.  

We consider the cube of resolutions of $D$, which is an $n$-dimensional cube with $2^n$ vertices, one for each state. Here we denote states by $\alpha\in\{0,1\}^n$, which is a binary sequence of length $n$ that indicates how each crossing has been resolved.
Let $r_{\alpha}$ and $k_{\alpha}$ be the number of $\oo$'s and cycles in $\alpha$, respectively. Let $\mathcal{C} ^{i,*}(D)$ be $\bigoplus V^{\otimes k_{\alpha}}\{r_{\alpha}+n_{+}-2n_{-}\}$, where we take the direct sum over all the states $\alpha$ with $r_{\alpha}=i+n_{-}$. Here $i$ is called \emph{horizontal}\index{horizontal grading} or  \emph{homological grading}\index{homological grading}.    

If $\mathcal{C}(D)=\bigoplus_{i}\mathcal{C}^i(D)$, then a \emph{horizontal grading shift}\index{horizontal grading shift} of $\mathcal{C}(D)$ by $l$ is defined as $\mathcal{C}(D)[l]=\bigoplus_{i}\mathcal{C}'^i(D)$, where $\mathcal{C}'^i(D)=\mathcal{C}^{i-l}(D)$.

We define the Khovanov complex as $\mathcal{C}Kh(D)=\bigoplus_{i,j}\mathcal{C}^{i,j}(D)$. To define the differential $d$, we introduce the product and coproduct maps. Note that henceforth we will suppress the symbol $\otimes$ in writing elements of $V^{\otimes k}$.
\begin{eqnarray*}
\Delta:V\to V\otimes V,&\ \ \ &m:V\otimes V\to V.\\
\oo\mapsto \oo X+X\oo&\ \ \ &\oo\oo\mapsto\oo     \\
X\mapsto XX&\ \ \ & \oo X\mapsto X  \\
&\ \ \ &     X\oo \mapsto X\\
&\ \ \ &    XX\mapsto 0
\end{eqnarray*}

We only define a map from a state $\alpha$ to a state $\alpha'$ if $\alpha'$ obtained from $\alpha$ by changing one $0$ to $1$. In that case, either two cycles of $\alpha$ merge into one cycle of $\alpha'$, or one cycle of $\alpha$ splits into two cycles of $\alpha'$. In the first case, we use the product map $m$, and in the second, we use the coproduct map $\Delta$. For all other cycles of $\alpha$, we apply the identity. In order to write down all the maps, we fix once and for all an enumeration of the cycles in each state, and these are not changed throughout the calculations. 

The last step is to assign negative signs to some of the maps. There are many ways to do that, but the homology groups for the different choices of signs
are all isomorphic. Here we follow the sign convention of \cite{Bar-natan-02}. 

Suppose we change $0$ to $1$ in the $m$-th spot to obtain $\alpha'$ from $\alpha$. In $\alpha$, we count how many $1$'s we have before the $m$-th spot. If it is an odd number, we assign a negative sign to the associated map. 

For a fixed $j$ the map $d^2:\mathcal{C}^{i,j}\to \mathcal{C}^{i+2,j}$ is zero and we obtain a bigraded homology theory denoted by $Kh^{*,*}(D)$. 

In \cite{Lee}, Lee constructs a new complex by modifying the maps $\Delta$ and $m$: 
\begin{eqnarray*}
\Delta':V\to V\otimes V,&\ \ \ &m':V\otimes V\to V.\\
\oo\mapsto \oo X+X\oo&\ \ \ &\oo\oo\mapsto\oo     \\
X\mapsto \oo\oo+XX&\ \ \ & \oo X\mapsto X  \\
&\ \ \ &     X\oo \mapsto X\\
&\ \ \ &    XX\mapsto \oo
\end{eqnarray*}

This results in a new homology theory called Lee homology and denoted $\text{Lee}(D)$. Notice the maps no longer preserve the quantum degree, thus Lee homology is only graded rather being bigraded. 

It turns out that $\text{Lee}(K)\cong \QQ\oplus\QQ$ for all knots, nevertheless as we will see the Lee homology contains a nontrivial and powerful invariant $s(K)$ called the \emph{Rasmussen invariant}\index{Rasmussen invariant}. This invariant was introduced by Rasmussen in \cite{Rasmussen}, and we briefly recall its definition. 

The quantum degree defines a decreasing filtration on $\mathcal{C}Kh(K)$. This induces a filtration on $\text{Lee}(K)$, 
$$H_{*}(\mathcal{C})=\mathcal{F}^nH_{*}(\mathcal{C})\supset \mathcal{F}^{n+1}H_{*}(\mathcal{C})\supset\ldots\supset \mathcal{F}^mH_{*}(\mathcal{C}).$$
For $x\in \text{Lee}(K)$, let $s(x)$ be the filtration degree of $x$, i.e.~$s(x)=k$ if $x\in\mathcal{F}^kH_{*}(\mathcal{C})$ but $x$ does not belong to $\mathcal{F}^{k+1}H_{*}(\mathcal{C})$.  We define 
\begin{eqnarray*}
s_{min}(K)&=&\min\{s(x)\in\text{Lee}(K)\mid x\neq0 \}, \\
s_{max}(K)&=&\max\{s(x)\in\text{Lee}(K)\mid x\neq0 \}.
\end{eqnarray*} 

Rasmussen proves that $s_{max}(K) = s_{min}(K)+2$ for all knots, and the \emph{Rasmussen invariant} is defined to be $s(K)=s_{min}(K)+1=s_{max}(K)-1$.

For a link $L$, the filtration on $\mathcal{C}Kh(L)$ induces a spectral sequence with $E_{0}$ term the Khovanov complex and $d_0=d_{Kh}$. It follows that the $E_1$ term is $Kh^{*,*}(L)$. For every $m$, $d_m=0$ unless $m$ is a multiple of $4$. As a result, for any $m\geq 0$, $E_{4m+1}\cong E_{4m+2}\cong E_{4m+3}\cong E_{4m+4}$. The $E_{\infty}$ page is isomorphic to the Lee homology. For a knot $K$, it has two copies of $\QQ$ which are placed on the $y$-axis. Their location indicates the filtration degree of the generators of the Lee homology. In particular the average of their $y$-coordinates is equal to $s(K)$. 

In \cite{Lee}, Lee proves that, for any alternating link $L$, its Khovanov homology  $Kh^{*,*}(L)$ is supported in the two lines $j=2i-\sigma(L)\pm1$. As a result, in the spectral sequence $d_{m}=0$ for every $m\geq 5$ and $E_{\infty}=E_5$. If $K$ is an alternating knot, then the $y$-coordinates of the two surviving copies of $\QQ$ are $-\sigma(K)\pm 1$. This implies that $s(K)=-\sigma(K)$.  

Lee homology is a functor, and if we have a cobordism $S$ between two links $L_0$ and $L_{1}$, then $S$ induces a map $\varphi_{S}':\text{Lee}(L_0)\to \text{Lee}(L_1)$ with filtration degree equal to $\chi(S)$. We will describe the map $\varphi'_S$ in a moment, but first notice that this implies that if $K$ is a knot, then $|s(K)|\leq 2g_{4}(K)$, where $g_{4}(K)$ denotes the classical 4-ball genus. The same inequality holds for the knot signature $\sigma(K)$, and Lee's theorem tells us that, for alternating knots, the Rasmussen invariant $s(K)$ gives the same bound on the $4$-ball genus as the knot signature. However, for non-alternating knots, it is no longer true that $s(K) = -\sigma(K)$, and sometimes the Rasmussen invariant provides a better bound. It should further be noted that Rasmussen's invariant gives a lower bound on the smooth 4-ball genus, whereas the knot signature gives a bound on the topological 4-ball genus.

\begin{example}
Let $K$ be the classical knot $9_{42}$. (For classical knots, we adopt the notation of  \cite{Knotinfo}.) Then this knot has Rasmussen invariant
 $s(K)=0$ and signature $\sigma(K)=2$. Thus the signature provides a better bound on the 4-ball genus than the Rasmussen invariant for this knot. On the other hand, for the knot $K=10_{132}$, we find that it has Rasmussen invariant $s(K)=-2$ and signature $\sigma(K)=0$. Thus, we see that the Rasmussen invariant gives a better bound on the $4$-ball genus in this case. 
\end{example}

We now describe the map $\varphi_{S}'$. Since any cobordism decomposes into a sequence of elementary cobordisms, it suffices to define $\varphi'_S$ for births, deaths, and saddles. In doing that, we will use the maps $\iota:\QQ\to V$ ($1\mapsto\oo$) and $\varepsilon:V\to\QQ$ ($\oo\mapsto 0$ and $X\mapsto 1$).

Note that an elementary cobordism is either a birth, a death, or a saddle. For a birth, we set $\varphi_{S}'=\iota$. For a death, we set $\varphi_{S}'=\varepsilon$. For a saddle $S$, $\varphi_{S}'$ is either $m'$ or $\Delta'$.



In general, the Rasmussen invariant is difficult to compute. However, the calculation simplifies for positive (or negative) knots, as we now explain.  
\begin{definition}
A link is called \emph{positive}\index{positive link} if it admits a diagram with only positive crossings. Similarly, a link is \emph{negative}\index{negative link} if it admits a diagram with only negative crossings. 
\end{definition}

If $K$ is a positive knot with diagram $D$ with $n$ positive crossings, then the Rasmussen invariant is given by $$s(K)=-k+n+1,$$
where $k$ is the number of cycles in the all $0$-smoothing state of $D$ \cite{Rasmussen}.

If $K$ is a negative knot with diagram $D$ with $n$ negative crossings, then the 
mirror image $D^{*}$ has $n$ positive crossings, and the all $0$-smoothing state of $D^{*}$ is the all $1$-smoothing state of $D$. Since $s(K^{*})=-k+n+1$, and since the Rasmussen invariant satisfies $s(K^{*})=-s(K)$ under taking mirror images, it follows that $s(K)=k-n-1$.

Next, we recall the definition of the Turaev genus for classical links from \cite{Tur}.
Let $D$ be a connected classical link diagram with $c(D)$ crossings, and suppose $s_0$ and $s_1$ are the all $0$ and all $1$ smoothing states, respectively. 

\begin{definition} \label{def-Turaev}
The \emph{Turaev genus} of the link diagram $D$ is defined by setting
$$g_{T}(D)=\frac{1}{2}(c(D)+2-|s_0(D)|-|s_1(D)|).$$
For a non-split classical link $L$, the Turaev genus, denoted $g_T(L)$, is defined to be
the minimal of $g_T(D)$ over all connected classical link diagrams $D$ for $L$.
\end{definition}

 For more on Turaev genus, see \cite{SurTur}.

\section{Khovanov Homology for Virtual Knots}\label{three}

In this section, we briefly introduce the Khovanov homology for virtual knots and links.

When one attempts to define a Khovanov theory for virtual knots the major problem is the presence of the \emph{single cycle smoothing}\index{single cycle smoothing} (see Figure \ref{single-cycle}). We need to assign a map to a single cycle smoothing, which we can do by assigning the zero map. In classical Khovanov theory, the signs of maps are chosen in a way to make each face of the cube of resolutions to be anti-commutative. Then this fact enables us to  define a differential $d$ satisfying $d^2=0$. For virtual knots, the existence of single cycle smoothings makes it more difficult to assign signs.

\vspace{.3cm} 
\begin{figure}[h!]
\centering\includegraphics[height=28mm]{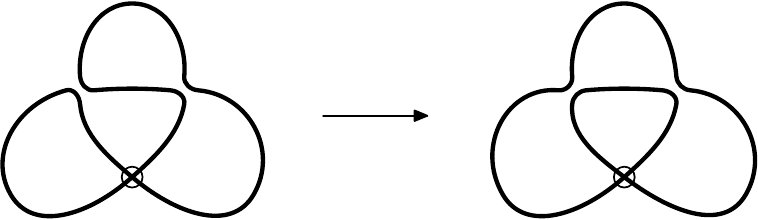}
\caption{A single cycle smoothing.} \label{single-cycle}
\end{figure}
\vspace{.3cm}

Tubbenhauer in \cite{Tubbenhauer}, used un-oriented TQFT's to define a Khovanov homology for virtual knots and links. 
In what follows, we describe Tubbenhauer's method. We will cover the combinatorial definitions. For a discussion about un-oriented TQFTs, see \cite{Tubbenhauer}.

Let $\QQ$ be the coefficient ring and $V=\QQ \oo\oplus \QQ X$. 
Start with a virtual link diagram $D$ with $n$ classical crossings. Resolve all the crossings in both ways to obtain $2^n$ states, leaving virtual crossings untouched. The Khovanov chain complex $\mathcal{C}(D)$ is defined as before, i.e.~we assign  $V^{\otimes k}$ to a state with $k$ components. The degree of each element and the grading shifts are defined as before. Whenever two vertices of an edge in the cube of resolutions have the same number of states, then that indicates the presence of a single cycle smoothing. In that case, we assign the zero
map to the edge. It remains to define the joining and splitting maps and the signs. 

Choose orientations for the cycles of each state. Although we can do this in an arbitrary way, to have less complicated maps at the end, we use a spanning tree argument. Choose a spanning tree for the cube of the resolution and start with the rightmost vertices and choose orientations for the cycles of corresponding states. Now remove those vertices and again choose orientations for the rightmost vertices, in a compatible way. That means we compare the two vertices which are joined by an edge of the spanning tree, and orient the cycles of the left vertex as follows. For cycles which are not involved in the join, split or the single cycle smoothing, orient each cycle of the left vertex exactly like the corresponding cycle in the right vertex. For other cycles try to orient them in a way to have the most compatibility. 

Choose an $x$-marker for each crossing and the corresponding $0$- and $1$-smoothings, as in Figure \ref{x-marker}. We choose either $x$ or $x'$ and notice that up to rotating the diagram and the corresponding states, there are only these two ways to assign the $x$-markers.

\vspace{.3cm}
\begin{figure}[h!]
\centering\includegraphics[height=26mm]{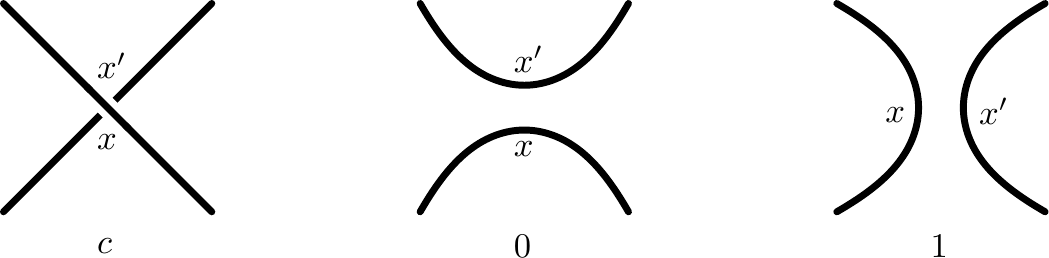}
\caption{An $x$-marker for a crossing and the corresponding smoothings.} \label{x-marker}
\end{figure}
\vspace{.3cm}

We define the sign of the non-zero maps as follows. By a spanning tree argument, number the cycles of each state. Suppose we have a joining map from a state $s$ to another state $s'$, and suppose $s$ has $m+1$ cycles. Let the joining map, merges the cycles number $a$ and $b$ in $s$ and the resulting cycle in $s'$ has number $c$, and let the cycle number $a$ has the $x$-marker. In an (m+1)-tuple, put $a$ first, then $b$ and then place the remaining numbers in an ascending order. Let $\tau_{1}$ be the permutation which takes $(1,2,\cdots,m+1)$ to this $(m+1)$-tuple. Now in an $m$-tuple, put $c$ first, and place the remaining numbers in an ascending order, and let $\tau_{2}$ be the permutation which takes $(1,2,\cdots,m)$ to this $m$-tuple. Define the sign of the joining map to be $\text{sign}(\tau_1)\text{sign}(\tau_2)$. The sign of the splitting map is defined similarly.

Next we define the maps. They are defined between the two vertices of an edge of the cube of resolutions. For each edge the smoothing of only one of the crossings is different, and we define a map from the state with $0$-smoothing to the state with $1$-smoothing. If a  cycle of the source state is disjoint from the smoothing change, assign the identity map to it, if its orientation agrees with the orientation of the corresponding cycle in the target state, otherwise assign negative of the identity map.

At a small neighborhood of the crossing, there are two parallel strings in each state. Notice that each cycle is oriented. Now if the map  looks like $\downarrow\uparrow\ \to\ \rightleftarrows$, we decorate the four strings in the source and target state with a $+$ sign, and we call this decoration \emph{standard}. We always rotate the states so the two  strings in the source state are vertical, and the two strings in the target state are horizontal. Then we compare the orientation of each string with the orientation of the corresponding one in the standard decoration, if they agree, decorate that string with a $+$ sign, otherwise decorate it with a $-$ sign. We record the possible cases in  Table \ref{decoration}.

\vspace{.3cm}
\begin{table}[h!]
\begin{center}
{\rowcolors{1}{white!80!white!50}{lightgray!70!lightgray!40}
\begin{tabular}{|c|c|c|c|}  \hline
 string & splitting map & string & joining map \\ \hline \hline 
$\downarrow\uparrow\ \to\ \rightleftarrows$   & $\Delta_{++}^{+}$ & $\downarrow\uparrow\ \to\ \rightleftarrows$ &  $m_{+}^{++}$  \\ \hline 
$\downarrow\uparrow\ \to\ \rightrightarrows$   & $\Delta_{-+}^{+}$ & $\uparrow\uparrow\ \to\ \rightleftarrows$ & $m_{+}^{-+}$    \\ \hline
$\downarrow\uparrow\ \to\ \leftleftarrows$   & $\Delta_{+-}^{+}$ & $\downarrow\downarrow\ \to\ \rightleftarrows$ & $m_{+}^{+-}$    \\ \hline
$\downarrow\uparrow\ \to\ \leftrightarrows$   & $\Delta_{--}^{+}$ & $\uparrow\downarrow\ \to\ \rightleftarrows$ & $m_{+}^{--}$    \\ \hline
$\uparrow\downarrow\ \to\ \rightleftarrows$   & $\Delta_{++}^{-}$ & $\downarrow\uparrow\ \to\ \leftrightarrows$ & $m_{-}^{++}$    \\ \hline
$\uparrow\downarrow\ \to\ \rightrightarrows$   & $\Delta_{-+}^{-}$ & $\uparrow\uparrow\ \to\ \leftrightarrows$ & $m_{-}^{-+}$    \\ \hline
$\uparrow\downarrow\ \to\ \leftleftarrows$   & $\Delta_{+-}^{-}$ & $\downarrow\downarrow\ \to\ \leftrightarrows$ & $m_{-}^{+-}$    \\ \hline
$\uparrow\downarrow\ \to\ \leftrightarrows$   & $\Delta_{--}^{-}$ & $\uparrow\downarrow\ \to\ \leftrightarrows$ & $m_{-}^{--}$    \\     
\hline 
\end{tabular}
}  
\end{center}
\caption{String decoration and corresponding maps.}\label{decoration}
\end{table}
\vspace{.3cm}

Other cases occur only when we have a single cycle smoothing. We describe the map $\Delta_{bc}^{a}(v)$ as follows. Multiply $v$ by $a$, apply $\Delta$, then multiply the first component of the resulting tensor product by $b$ and the second component by $c$. Notice that the first component of the tensor product, corresponds to the lower string or the string with the $x$-marker on it. Similarly, we define the map $m_{a}^{bc}$. First multiply the first component of the tensor product by $b$ and the second component by $c$, then apply $m$, at the end multiply the result by $a$. 

\begin{remark}
We know that every checkerboard colorable diagram admits a source-sink orientation (\cite[Proposition 6]{Kamada-skein}). We can use this orientation to make all the decorations standard. In that case we only need the maps $\Delta_{++}^{+}$ and $m_{+}^{++}$. 
\end{remark}

\begin{example}\label{ex37}
We compute the Khovanov homology for the virtual knot $K=3.7$; here the decimal number refers to the virtual knots in Green's tabulation \cite{Green}.  Figure \ref{3-7} is a diagram for $K$ and Figure \ref{resolution} is the cube of resolutions. 

\vspace{.3cm}
\begin{figure}[h!]
\centering\includegraphics[height=32mm]{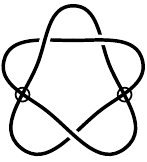}
\caption{The virtual knot $3.7$.} \label{3-7}
\end{figure}
\vspace{.3cm}

 We enumerate the components of a state in a way that the one which has more $x$-markers in it be the first component. All the $m$ maps are $m_{+}^{++}$, and $\Delta$ maps are $\Delta_{++}^{+}$. A red arrow means the associated map has negative sign. All the maps are a single  splitting or joining map except for the state which has 3 components in it. For this state the incoming map is $\Delta\otimes id$, and for the outgoing maps, the upper one is $\varphi$ defined as $\varphi(a,b,c)=-m(a,c)\otimes b$, and the lower one is $-id\otimes m$.

\vspace{.3cm}
\begin{figure}[h!]
\centering\includegraphics[height=94mm]{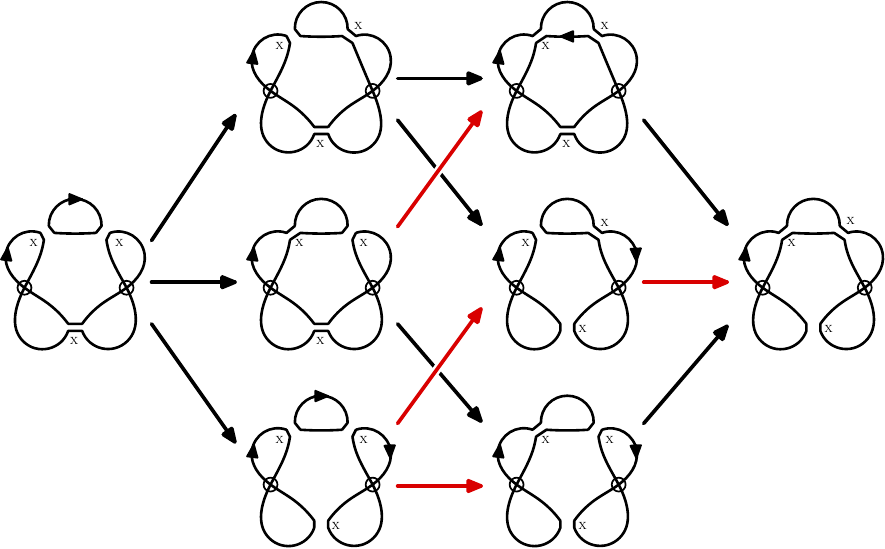}
\caption{The cube of resolutions for $K=3.7$.} \label{resolution}
\end{figure}
\vspace{.3cm}

The Khovanov complex is as follows:
$$V\otimes V\{-3\}\to V\oplus V\oplus \left(V^{\otimes 3}\right)\{-2\}\to
\left(V^{\otimes 2}\right)\oplus\left(V^{\otimes 2}\right)\oplus\left(V^{\otimes 2}\right)\{-1\}
\to V.$$

We record the basis elements of the chain complex in  Table \ref{basis elements}.

\vspace{.3cm}
\begin{table}[h!]
\small
\begin{center}
{\rowcolors{1}{white!80!white!50}{lightgray!70!lightgray!40}
\begin{tabular}{|c|c|c|c|c|}  \hline
$j\setminus i$ & $-2$ & $-1$ &0&1 \\ \hline \hline 
   &  &    & $(\oo\oo,0,0)$ &  \\ 
   1          &&$(0,0,\oo\oo\oo)$& $(0,\oo\oo,0)$ &                 $\oo$                  \\  
             &&&  $(0,0,\oo\oo)$ &                                   \\    \hline
  &  & $(\oo ,0,0)$ & $(\oo X,0,0)$ &  \\ 
  &  & $(0,\oo,0)$ & $(X\oo,0,0)$ &  \\ 
 $-1$  & $\oo\oo$ & $(0,0,\oo\oo X)$ & $(0,\oo X,0)$ & $X$ \\ 
    &  & $(0,0,\oo X\oo)$ &$(0,X\oo,0)$  &  \\ 
     &  &$(0,0,X\oo\oo)$  &$(0,0,\oo X)$  &  \\ 
      &  &  & $(0,0,X\oo)$ &  \\ 
 \hline
&  & $(X,0,0)$ &  &  \\ 
& $\oo X$ & $(0,X,0)$ &$(XX,0,0)$  &  \\
$-3$&  &$(0,0,\oo XX)$  &$(0,XX,0)$  &  \\
&$X\oo $  & $(0,0,X\oo X)$ & $(0,0,XX)$ &  \\
&  & $(0,0,XX\oo)$ &  &  \\ 
 \hline
 $-5$ & $XX$ & $(0,0,XXX)$ &  & \\
\hline
\end{tabular}
}  
\end{center}
\caption{The basis elements for the chain complex.}\label{basis elements}
\end{table}
\vspace{.3cm}

The image of each basis element is in Table \ref{image of the basis elements}.

\vspace{.3cm}
\begin{table}[h!]
\small
\begin{center}
{\rowcolors{1}{white!80!white!50}{lightgray!70!lightgray!40}
\begin{tabular}{|c|c|c|c|c|}  \hline
$j\setminus i$  & $-2$ & $-1$ &0&1 \\ \hline \hline 
   &  &    & $ \oo$ &  \\ 
   1          &&$(0,-\oo\oo,-\oo\oo)$& $ -\oo$ &                                $ 0$   \\  
             &&&  $ \oo$ &                                   \\    \hline
  &  & $(\oo X+X\oo,\oo X+X\oo,0)$ & $ X$ &  \\ 
  &  & $ (-\oo X-X\oo,0,\oo X+X\oo)$ & $ X$ &  \\ 
 $-1$  & $(\oo,\oo,\oo X\oo+X\oo\oo)$ & $ (0,-X\oo,-\oo X)$ & $ -X$ & $ 0$ \\ 
    &  & $ (0,-\oo X,-\oo X)$ &$ -X$  &  \\ 
     &  &$ (0,-X\oo,-X\oo)$  &$ X$  &  \\ 
      &  &  & $ X$ &  \\ 
 \hline
&  & $ (XX,XX,0)$ &  &  \\ 
& $ (X,X,\oo XX+X\oo X)$ & $ (-XX,0,XX)$ &$ 0$  &  \\
$-3$&  &$ (0,-XX,0)$  &$ 0$  &  \\
&$(X,X,XX\oo)$  & $ (0,0,-XX)$ & $ 0$ &  \\
&  & $ (0,-XX,-XX)$ &  &  \\ 
 \hline
 $-5$ & $ (0,0,XXX)$ & $ (0,0,0)$ &  & \\
\hline 
\end{tabular}
}  
\end{center}
\caption{The image of the basis elements.}\label{image of the basis elements}
\end{table}
\vspace{.3cm}

It is easy to check $d^2=0$. When we take the homology, two copies of $\QQ$ survive, both in homological degree $0$, one in quantum degree $1$ and the other in quantum degree $-1$. Therefore the Khovanov homology of $K$ is isomorphic to the Khovanov homology of the unknot.  
\end{example}  

In \cite{DKK}, Dye, Kaestner and Kauffman define Lee homology and the Rasmussen invariant for virtual knots, and they show that the Rasmussen invariant is an invariant of virtual knot concordance.

\begin{remark}
Khovanov homology and the Rasmussen invariant  are invariants of unoriented virtual knots. If $D$ is a virtual knot diagram, then under mirror symmetry, we have
$$Kh^{i,j}(D^{*})=Kh^{i,j}(D^{\dag})=Kh^{-i,-j}(D)\quad \text{ and } \quad s(D^{*}) = s(D^{\dag}) = -s(D).$$
The statement about $D^\dag$ follows from the one about $D^*$ since $D^\dag$ is obtained by applying {\bf or} to all the crossings in $D^*$. More generally, applying {\bf or} to a crossing does not change the cube of resolutions (see \cite{DKK}), nor does it change any of the quantities needed to compute the Khovanov homology, such as the number of positive and negative crossings. Hence the Khovanov homology and the Rasmussen invariant are unchanged under the {\bf or} move. 
\end{remark}

\begin{remark}
Connected sum is not a well-defined operation on virtual knots; it depends on the diagrams used and the choice of where to form the connected sum. The Rasmussen invariant is independent of these choices, and it is, in fact,
additive under connected sum. For a proof, we refer the reader to  \cite{Rushworth}. This fact and invariance under concordance imply that it induces a homomorphism from the virtual concordance group to $\ZZ$.
\end{remark}

\begin{example} Table \ref{ras-table} lists the Rasmussen invariant for the alternating virtual knots up to six crossings. The three virtual knots 6.90115, 6.90150 and 6.90170 appearing in Figure \ref{fig-6-90170} all have Rasmussen invariant equal to $-2$, and as a result we conclude that none of these virtual knots are slice. 

In \cite{BCG17}, Boden et al.~define slice obstructions in terms of signatures of symmetrized Seifert matrices for almost classical knots, and as an application, they show that neither 6.90115 nor 6.90150 are slice. The Rasmussen invariant provides an alternative proof of non-sliceness for these two almost classical knots, and it also gives a new result by showing that 6.90170 is not slice.

Since each of these knots can be unknotted using two crossing changes, it follows that their slice genus satisfies $1 \leq g_{s}(K) \leq  2$.
\end{example} 

\begin{figure}[h!]
\includegraphics[scale=1.6]{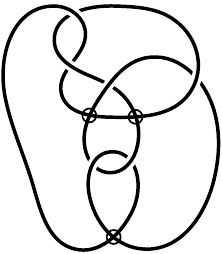} \hspace{1.2cm} 
\includegraphics[scale=1.45]{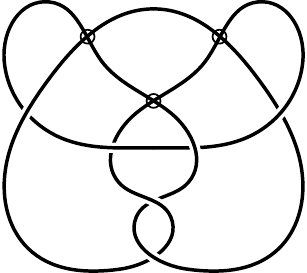} \hspace{1.2cm} 
\centering\includegraphics[scale=1.6]{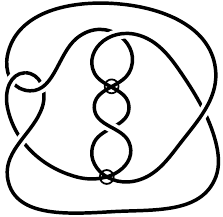}
\caption{The alternating virtual knots $6.90115,$ $6.90150$ and $6.90170$, from left to right.} \label{fig-6-90170}
\end{figure}
\vspace{.3cm}


\section{Khovanov Homology and Alternating Virtual Links}\label{four}

Let $D$ be a checkerboard virtual link diagram. Apply {\bf or} to all crossings with $\eta=-1$. The result is a diagram in which $\eta = +1$ for each crossing. Hence, by  Lemma \ref{eta}, the new diagram, which we call $D_{alt}$, is alternating. The diagrams $D$ and $D_{alt}$ have isomorphic Khovanov homology groups.

In particular, starting with any classical diagram, we can change it to an alternating virtual diagram with the same Khovanov homology. We can do the same, starting with any checkerboard colorable diagram.

Following  \cite{Lee}, we seek a relation between the Rasmussen invariant and the signatures of alternating virtual knots. If $i$ is the homological degree and $j$ is the quantum degree for Khovanov homology, then $H$-thinness for classical alternating knots means, $j=2i-\sigma\pm 1$, where $\sigma$ is the signature. This implies that $s=-\sigma$, where $s$ is the Rasmussen invariant. 
On the other hand, not all virtual alternating knots are $H$-thin. For example the Khovanov polynomial for the knot $K=5.2426$ depicted in Figure \ref{fig-5-2426} is as follows: 
$$\frac{1}{q^{11}t^{3}}+\frac{1}{q^{9}t^{3}}+\frac{1}{q^{7}t^{2}}+\frac{1}{q^{5}t^{2}}+\frac{1}{q^{5}}+\frac{1}{q^{3}},$$ which is supported in three lines $j=2i-1$, $j=2i-3$ and $j=2i-5$. Notice that from Table \ref{ras-table} $(\sigma_{\xi^{*}},\sigma_{\xi})=(2,4),$ and we can write the three lines as:
$$j=2i-\sigma_{\xi^*}+1,j=2i-\sigma_{\xi^*}-1,j=2i-\sigma_{\xi}-1.$$

In fact, instead of $H$-thinness we have the following theorem. Here the coloring $\xi$ is the one for which every  crossing of $D$ has $\eta=-1$.

\begin{theorem}\label{prop}
If $D$ is a connected alternating virtual link diagram with genus $g$ and signatures $\sigma_{\xi},\sigma_{\xi^*}$, then its Khovanov homology is supported in $g+2$ lines: $$j=2i-\sigma_{\xi^*}+1,j=2i-\sigma_{\xi^*}-1,\ldots,j=2i-\sigma_{\xi}-1.$$
\end{theorem}

\begin{proof}
Following \cite{Lee}, we apply induction on the number of crossings. The base case is trivial. Let $D$ be an alternating virtual link diagram with $n$ crossings. $0$ and $1$ smooth the last crossing to obtain $D(*0)$ and $D(*1)$, respectively. We can easily see that they are alternating link diagrams. Shift the Khovanov complex by $n_{-}$ horizontally, and $2n_{-}-n_{+}$ vertically. Denote the resulting complex by $\wbar{C}(D)$ and its homology by $\wbar{H}(D)$. We denote this shift by $\wbar{C}(D)=C(D)[n_{-}]\{2n_{-}-n_{+}\}$. We have the following short exact sequence: 
$$0\to \wbar{C}(D(*1))[+1]\{+1\}\to \wbar{C}(D)\to \wbar{C}(D(*0))\to 0,$$  which gives a long exact sequence involving $\wbar{H}(D), \wbar{H}(D(*0))$ and
$\wbar{H}(D(*1))[+1]\{+1\}$, which encodes $\wbar{H}(D)$ is supported on $\wbar{H}(D(*0))$ and $\wbar{H}(D(*1))$. 

It suffices to show that $\wbar{H}(D)$ is supported in $g+2$ lines with $y$-intercepts of $$-|s_{\partial}|+2,-|s_{\partial}|,\cdots,-|s_{\partial}|-2g,$$ because after shifting back $\wbar{H}(D)$, the result follows. 

The all $0$ state of $D$ is the same as the all $0$ state of $D(*0)$. Also the all $1$ state of $D$ is the same as the all $1$ state of $D(*1)$. In the all $0$ state of $D$, if we change the resolution of the last crossing from a $0$-smoothing to a $1$-smoothing, we obtain the all $0$ state for $D(*1)$. Similarly, in the all $1$ state of $D$, if we change the resolution of the last crossing from a $1$-smoothing to a $0$-smoothing, we obtain  the all $1$ state for $D(*0)$. 

These three diagrams, all have the boundary property. $D(*0)$ and $D(*1)$, both have $n-1$ crossings. Thus we have: 
\begin{eqnarray*}
|s_{\partial}(D)|+|\wbar{s}_{\partial}(D)|&=&n+2-2g(D),\\
|s_{\partial}(D(*0))|+|\wbar{s}_{\partial}(D(*0))|&=&n+1-2g(D(*0)),\\
|s_{\partial}(D(*1))|+|\wbar{s}_{\partial}(D(*1))|&=&n+1-2g(D(*1)).
\end{eqnarray*}
Using the above observations, we can rewrite the last two equations as: 
\begin{eqnarray*}
|s_{\partial}(D)|+|\wbar{s}_{\partial}(D(*0))|&=&n+1-2g(D(*0)),\\
|s_{\partial}(D(*1))|+|\wbar{s}_{\partial}(D)|&=&n+1-2g(D(*1)).
\end{eqnarray*}

Since the genus is an integer, the first equation implies that $|\wbar{s}_{\partial}(D(*0))|$ cannot be equal to $|\wbar{s}_{\partial}(D)|$, so it is either one more, or one less. Similarly, $|s_{\partial}(D(*1))|$ is either one more, or one less than $|s_{\partial}(D)|$. Thus we have four different cases: 

\noindent {\bf Case 1:} 
$|\wbar{s}_{\partial}(D(*0))|=|\wbar{s}_{\partial}(D)|-1\ ,\ |s_{\partial}(D(*1))|=|s_{\partial}(D)|-1\ \to\ g(D)=g(D(*0))=g(D(*1)).$ 

We use the induction hypothesis. Since $|s_{\partial}(D(*0))|=|s_{\partial}(D)|$ and $g(D(*0))=g(D)$, the  $y$-intercepts of the lines for $D(*0)$, are: 

$$-|s_{\partial}(D)|+2,-|s_{\partial}(D)|,\cdots,-|s_{\partial}(D)|-2g(D).$$

The $y$-intercepts of the lines for  $D(*1)[+1]\{+1\}$ are the $y$-intercepts of the lines for $D(*1)$ minus $1$. Since  $|s_{\partial}(D(*1))|=|s_{\partial}(D)|-1$, the $y$-intercepts of the lines for $D(*1)[+1]\{+1\}$ and $D(*0)$ agree, and they are precisely the numbers that we are looking for. Thus the result follows in this case.

\noindent {\bf Case 2:} 
$|\wbar{s}_{\partial}(D(*0))|=|\wbar{s}_{\partial}(D)|+1\ ,\ |s_{\partial}(D(*1))|=|s_{\partial}(D)|-1\ \to\ g(D)=g(D(*0))+1=g(D(*1)).$

In this case, there are $g(D)+1$ lines for $D(*0)$, and their $y$-intercepts  are: 
$$-|s_{\partial}(D)|+2,-|s_{\partial}(D)|,\cdots,-|s_{\partial}(D)|-2g(D)+2.$$

On the other hand for $D(*1)$, the $y$-intercepts are as before. Hence the union of the supports of $D(*0)$ and $D(*1)[+1]\{+1\}$ is again the desired $g(D)+2$ lines.

\noindent {\bf Case 3:} 
$|\wbar{s}_{\partial}(D(*0))|=|\wbar{s}_{\partial}(D)|-1\ ,\ |s_{\partial}(D(*1))|=|s_{\partial}(D)|+1\ \to\ g(D)=g(D(*0))=g(D(*1))+1.$

In this case, there are $g(D)+1$ lines for $D(*1)[+1]\{+1\}$, and their $y$-intercepts are: 
$$-|s_{\partial}(D)|,-|s_{\partial}(D)|,\cdots,-|s_{\partial}(D)|-2g(D).$$

For $D(*0)$, we have the same $g(D)+2$ line, as in case 1. As before, their union is the $g(D)+2$ lines with the desired $y$-intercepts.

\noindent {\bf Case 4:} 
$|\wbar{s}_{\partial}(D(*0))|=|\wbar{s}_{\partial}(D)|+1\ ,\ |s_{\partial}(D(*1))|=|s_{\partial}(D)|+1\ \to\ g(D)=g(D(*0))+1=g(D(*1))+1$.

Combining case 2 and 3, we see that the result follows. 
\end{proof}

\begin{corollary} 
Classical alternating links are $H$-thin. 
\end{corollary}

\begin{corollary}
If $D$ is a connected alternating virtual link diagram with genus $g$ and signatures $\sigma_{\xi},\sigma_{\xi^*}$, then 
$$-\sigma_{\xi}\leq s(D)\leq -\sigma_{\xi^{*}}.$$
\end{corollary}

\begin{proof}
This follows from the previous theorem, and Lee's spectral sequence. 
\end{proof}

\begin{example}
For the classical knot $K=9_{42}^{*}$ shown on the left of Figure \ref{fig-9-42-alt}, the Khovanov polynomial is as follows
$$\frac{1}{q^7t^4}+ \frac{1}{q^3t^3}+ \frac{1}{q^3t^2}+ \frac{1}{qt}+\frac{q}{t}+ \frac{1}{q}+q+ q^3+ q^3t+ q^7t^2,$$
which is supported in three lines. 

Observe that, given a 2-strand classical tangle with $n$ half-twists, applying the 
{\bf or} move to each of the crossings has the effect of adding two virtual crossings 
 at either end of the tangle (see Figure \ref{last}).

\begin{figure}[h!]

\includegraphics[height=16mm]{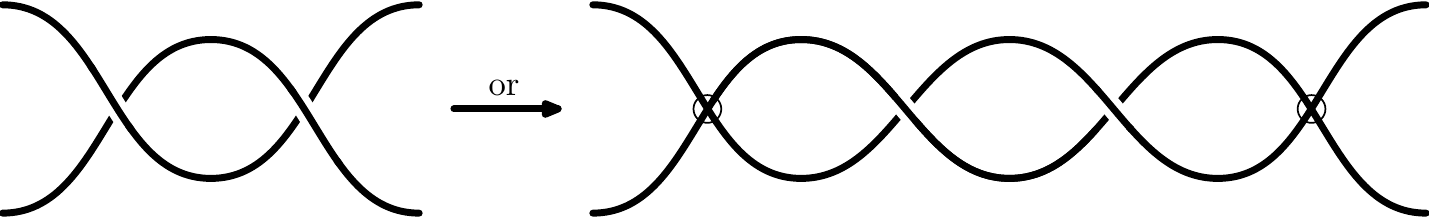}\vspace{1cm}

\includegraphics[height=16mm]{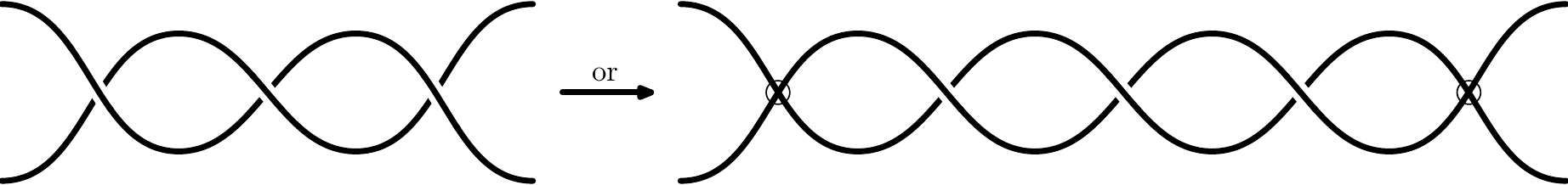}
\caption{The effect of applying {\bf or} to a 2-strand classical tangle with $n$ half-twists.}
\label{last}
\end{figure}
\vspace{.3cm}

Let $\xi$ be the coloring in which the unbounded region is white. 
Then the five crossings in the top half of $9_{42}^*$ of Figure \ref{fig-9-42-alt} have $\eta =-1$ and the four crossings in the bottom half have $\eta=1.$
The five crossings with $\eta=-1$ occur in two 2-strand classical tangle, one with three positive crossings and the other with two negative crossings. 
Thus, the above observation implies that applying the {\bf or} move to the five crossings with $\eta=-1$ results in the alternating virtual knot shown on the right of Figure \ref{fig-9-42-alt}.  
Then $(\sigma_{\xi^{*}},\sigma_{\xi})=(-2,0)$, and by Theorem \ref{prop}, its Khovanov homology, which coincides with the Khovanov homology of $K=9_{42}^{*}$ is supported in the following three lines:
$$j=2i+3,j=2i+1,j=2i-1.$$
\end{example}

\begin{figure}[h!]
\centering\includegraphics[height=46mm]{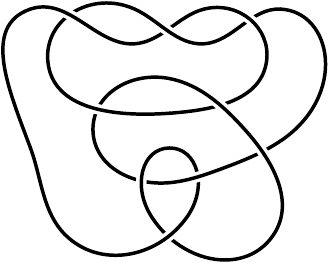}\hspace{1.2cm}
\includegraphics[height=46mm]{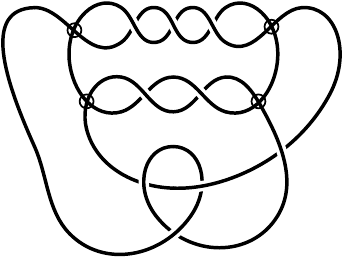}
\caption{The classical knot $9_{42}^{*}$ (left) and the virtual knot $(9_{42}^{*})_{alt}$ (right).}
\label{fig-9-42-alt}
\end{figure}
\vspace{.3cm}

Similar to the Turaev genus (cf.~Definition \ref{def-Turaev}), we have the following definition.

\begin{definition}
Let $K$ be a non-split checkerboard colorable virtual link. The \emph{alternating genus} of $K$ is 
$$g_{alt}(K)=\min\{g(D_{alt}) \mid D\ \text{is\ a\ checkerboard\ diagram\ for\ }K\},$$
where $g(D_{alt})$ is the supporting genus of $D_{alt}$.
\end{definition}

\begin{lemma} \label{lem-alt-Turaev}
If $L$ is a non-split classical link, then $g_{alt}(L)\leq g_T(L)$.
\end{lemma}

\begin{proof}
Let $D$ be a classical diagram for $L$. We obtain $s_{\partial}$ by $0$-smoothing crossings $c$ with $\eta(c)=+1$, and $1$-smoothing crossings $c$ with $\eta(c)=-1$. To obtain $D_{alt}$ we apply {\bf or} exactly to crossings $c$ with $\eta(c)=-1$. In $D_{alt}$ black and white disks are obtained by switching $0$ and $1$ smoothing and vice versa for crossings $c$ with $\eta(c)=-1$. This means $s_{\partial}(D_{alt})=s_0(D)$ and $\wbar{s}_{\partial}(D_{alt})=s_1(D)$, and we have: 
\begin{eqnarray*}
g(D_{alt})&=&\frac{1}{2}\left(c(D_{alt})+2-|s_{\partial}(D_{alt})|-|\wbar{s}_{\partial}(D_{alt})|\right),\\
&=&\frac{1}{2}\left(c(D)+2-|s_{0}(D)|-|s_{1}(D)|\right)=g_T(D).
\end{eqnarray*}
This shows $g_{alt}(L)\leq g_T(L)$.  
\end{proof}

The following corollary is immediate, and is a generalization of Corollary 3.1 in \cite{Kofman07}, which 
was first obtained by Manturov in \cite{Man}.

\begin{corollary}\label{cor3}
For any checkerboard colorable (in particular, classical) link $K$, the Khovanov homology 
of $K$ is supported in $g_{alt}+2$ lines, i.e.~the homological  width of the Khovanov homology is less than or equal to $g_{alt}+2$.   
\end{corollary}

Here we have another proof for Corollary 3.1 in \cite{Kofman07}.

\begin{corollary}
For any classical non-split link L with Turaev genus $g_T(L)$, the thickness of the
(unreduced) Khovanov homology of K is less than or equal to $g_T(L)+2$.
\end{corollary}

\begin{proof}
We have $g_{alt}(L)\leq g_T(L)$, and $g_{alt}(L)+2$ is an upper bound for the thickness of the Khovanov homology. The result is immediate. 
\end{proof}

\begin{lemma}
Suppose $D$ is a positive alternating virtual knot. Then $s(D)=-\sigma_{\xi^*}(D)$. 
\end{lemma}

\begin{proof}
Since $D$ is alternating, $\sigma_{\xi^*}=\beta-1-n_{+}$, where $\beta$ is the number of all $0$-smoothing state (see \cite{Karimi}). For any positive knot $K$, we have $s(K)=1-\beta+n_{+}$ (see \cite{DKK}). 
\end{proof}

For a negative virtual knot (for example $K=5.2426$ depicted in Figure \ref{fig-5-2426}), the Rasmussen invariant and the signatures are the negatives of the corresponding invariants for its vertical mirror image (a positive virtual knot). It follows that $s(K)=-\sigma_{\xi}(K)$. 
In general it is not true that the Rasmussen invariant is the negative of one of the signatures for alternating virtual knots. For example, the virtual knot $5.2427$ is alternating (see  Figure \ref{fig-5-2426}), has Rasmussen invariant $s(K)=-2$, and signatures $\sigma_\xi(K)=4$ and $\sigma_{\xi^*}(K)=0$.

\begin{figure}[h!]
\centering\includegraphics[height=46mm]{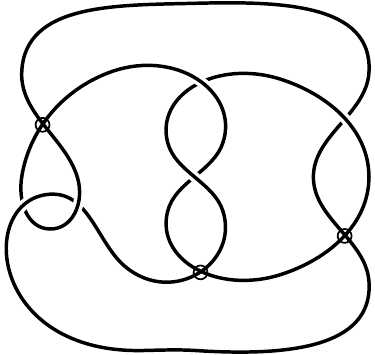}
\qquad \qquad 
\includegraphics[height=46mm]{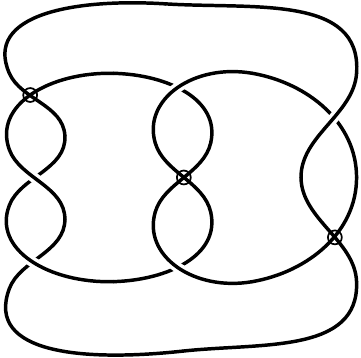}
\caption{Alternating virtual knots $5.2426$ (left) and $5.2427$ (right).} \label{fig-5-2426}
\end{figure}
\vspace{.3cm}

Ordinarily, for Khovanov homology of virtual links, one assigns the zero map to each single cycle smoothing. Alternative theories can be constructed using different maps for the single cycle smoothings, and as virtual knot homologies, these are generally stronger than the usual Khovanov homology for virtual knots. For example, in \cite{Rushworth} Rushworth uses this approach to define a variant theory called doubled Khovanov homology. For virtual links whose cube of resolutions has no single cycle smoothings,  the doubled Khovanov homology is the direct sum of two copies of ordinary Khovanov homology. In that case, the doubled Khovanov homology is completely determined by the ordinary Khovanov homology and thus it contains no new information.

We will show that when the underlying virtual link diagram is checkerboard colorable, there are no single cycle smoothings in its cube of resolutions. This was first proved by Rushworth  \cite{Rushworth}, and here we provide a different proof.


\begin{lemma}
Let $D$ be an alternating link diagram, and $s_{\partial}$ be the all $0$-smoothing state, and $\wbar{s}_{\partial}$ the all $1$-smoothing state. If we change one $0$-smoothing to obtain the state $s$, the number of components of $s_{\partial}$ and $s$, differs by one. A similar result holds for $\wbar{s}_{\partial}$.
\end{lemma}    

\begin{proof}
Assume we change the smoothing in the last crossing. We consider $D(*1)$, which is an alternating diagram and has the boundary property. If $D$ has $c$ crossings, and $g$ and $g_{1}$ are the genera for $D$ and $D(*1)$ respectively, we have:
\begin{eqnarray*}
|s_{\partial}|+|\wbar{s}_{\partial}| &=& c+2-2g,\\
|s|+|\wbar{s}_{\partial}| &=& c-1+2-2g_{1},\\
|s_{\partial}|-|s| &=& 1+2(g_{1}-g).
\end{eqnarray*}

Thus the difference is an odd number, and the result follows. The proof for the other case is similar. 
\end{proof}

\begin{proposition}
Let $D$ be an alternating link diagram. Then there is no single cycle smoothing in the cube of resolutions for $D$.
\end{proposition}

\begin{proof}
Assume we change a $0$-smoothing of the state $s$ to a $1$-smoothing at the crossing $c_{i}$. If for all the other crossings, we have $0$-smoothing in $s$, then this is the previous lemma. Otherwise, we apply {\bf sc} to the crossings of $D$, which have been resolved to $1$-smoothings in $s$. Call the new diagram $D'$. Since the state $s$ is the all $0$-smoothing state for $D'$, the result follows from the previous lemma.   
\end{proof}

\begin{proposition}
Let $D$ be a checkerboard colorable link diagram. Then there is no single cycle smoothing in the cube of resolutions for $D$.
\end{proposition}

\begin{proof}
Assume we change one $0$-smoothing of the state $s$ to a $1$-smoothing at the crossing $c_{i}$, and call the resulting state $s'$. First we consider $D_{alt}$. Let $C'=\{c_{i_1},\ldots,c_{i_k}\}$ be the set of crossings of $D$ which are changed to obtain $D_{alt}$. There are two cases. If $c_{i}$ does not belong to $C'$, then the edge with vertices $s$ and $s'$ corresponds exactly to an edge in the cube of resolutions for $D_{alt}$, and the result follows.

If $c_{i}\in C'$, then the same thing happens. The only difference is the direction of the map in $D_{alt}$ is reversed, going from $s'$ to $s$. The result still holds.
\end{proof}

\begin{corollary}
If $D$ is a checkerboard colorable link diagram, then the doubled Khovanov homology for $D$ is the direct sum of two copies of the ordinary Khovanov homology for $D$. 
\end{corollary}

We have calculated the Rasmussen invariant for all alternating virtual knots up to 6 crossings, and the results are listed in Table \ref{ras-table}. Note that the
Rasmussen invariants are calculated for virtual knots up to 4 crossings in \cite{rush}.

In \cite{Boden2019}, Boden and Chrisman list the number of all virtual knots up to 6 crossings with unknown slice status (cf.~\cite{MTables}). The following example concerns three non-alternating virtual knots whose slice status was previously unknown.
We use Rasmussen invariants to show they are not slice and deduce that they have slice genus equal to one. 

\begin{example}
Consider the three non-alternating virtual knots 6.31460, 6.52378, and 6.66907 depicted in Figure \ref{Fig:non-slice}. They all have isomorphic Khovanov homology, and the Khovanov polynomial is as follows
$$\frac{1}{q^{3}t^{2}}+\frac{1}{qt}+\frac{q}{t}+2+q+2q^{3}+2q^{2}t+2q^{4}t+q^{3}t^{2}+2q^{6}t^{2}+q^{7}t^{3}.$$ From this polynomial, it is not difficult to see that the two surviving copies of $\QQ$ in Lee's spectral sequence are in degrees $1$ and $3$. It follows these three knots have Rasmussen invariant equal to $2$.

Since each of these virtual knots has nonzero Rasmussen invariant, none of them are slice. Further, notice that, for each of 6.31460, 6.52378, and 6.66907, performing a crossing change to one of the crossings in the clasp produces a diagram of the unknot. Since a crossing change can be achieved in a genus one cobordism, it follows that each of the three virtual knots has slice genus equal to one.

\vspace{.3cm} 
\begin{figure}[h!]
\centering\includegraphics[scale=1.30]{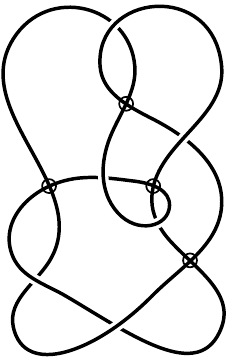}
\hspace{1.2cm} 
\includegraphics[scale=1.3]{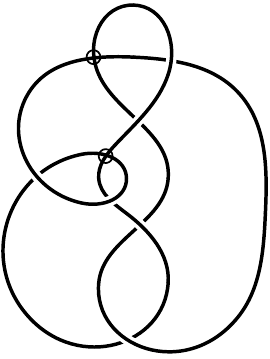}
\hspace{1.2cm} 
\includegraphics[scale=1.3]{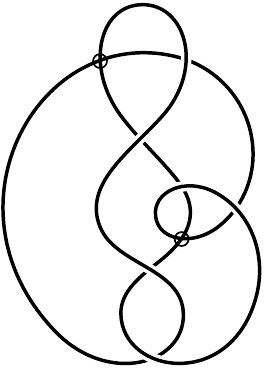}
\caption{The non-alternating virtual knots $6.31460,$ $6.52378$ and $6.66907,$ from left to right.}
\label{Fig:non-slice}
\end{figure}
\vspace{.3cm}


\end{example}
%
%
%


\section{Alternating Virtual Knots and Their Invariants}\label{five}

The signatures  are computed using a Mathematica program written by
Micah Chrisman, and the Khovanov homology (unlisted) is computed using online Mathematica program written by Daniel Tubbenhauer. The Rasmussen invariants are then computed by hand using Lee's spectral sequence. Boldface font is used to indicate that the knot is classical. Here the decimal numbers refer to the virtual knots in Green's tabulation \cite{Green}. For a list of the associated Gauss words, see \cite{Karimi}.  

In the following example, we outline how to calculate the Rasmussen invariant of a virtual knot once its Khovanov homology has been determined.

\begin{example}
For the alternating virtual knot $K=6.90170$ depicted in Figure \ref{fig-6-90170}, the Khovanov homology is recorded in Table \ref{690170}. In the spectral sequence, $E_{\infty}\cong E_5$, and $E_4\cong E_2$. This fact dictates the exact location of the one nontrivial $d_4$, which is from the $(-3,-9)$-entry to the $(-2,-5)$-entry of Table \ref{690170}.  Now the surviving copies of $\QQ$ in $E_{\infty}$, are in the entries $(0,-1)$ and $(0,-3)$, which shows that this knot has Rasmussen invariant $s=-2$.

\vspace{.3cm}
\begin{table}[h!]
\begin{center}
{\rowcolors{1}{white!80!white!50}{lightgray!70!lightgray!40}
\begin{tabular}{|c|c|c|c|c|}  \hline
$j\setminus i$ & $-3$ & $-2$ & $-1$ & 0 \\ \hline \hline 
 $-1$  &  &    &  &  $\QQ$\\
 \hline 
       $-3$      &&&& $\QQ$\\
       \hline
             $-5$  &  &$\QQ$    &  &  \\
             \hline
             $-7$  &  &    &  &  \\
             \hline
             $-9$  &$\QQ$  &    &  &  \\
\hline
\end{tabular}
}  
\end{center}
\caption{The Khovanov homology for the alternating virtual knot $K=6.90170$.}\label{690170}
\end{table}
\vspace{.3cm}

\end{example}
 
\bigskip \noindent
{\bf Acknowledgements.} I would like to thank my adviser, Hans U. Boden, for all of his support. I would also like to thank Micah Chrisman for providing the Mathematica package used to compute the signatures. 

\newpage
\begin{table}[H] 
\tiny
 \begin{tabular}{cccc}

{\rowcolors{1}{white!80!white!50}{white!70!white!40}
\begin{tabular}{|l|c|c|l|}
\hline
\tiny

{Virtual} &  \tiny {Signatures}    & \tiny{Rasmussen} & \tiny{Khovanov}  \\ 
\ \tiny {Knot} &   \tiny    $(\sigma_\xi^*,\sigma_{\xi})$            & \tiny{Invariant}&\tiny{Polynomial}
\\ \hline \hline
{\bf 3.6} & $(2,2)$  &   $-2$& $1/q^{9}t^{3}+1/q^{5}t^{2}+1/q^{3}+1/q$ \\ \hline
3.7& $(0,2)$  &   0  &  $1/q+q$ \\ \hline
4.105& $(0,2)$ &$-2$ &$1/q^{9}t^{3}+1/q^{5}t^{2}+1/q^{3}+1/q$ \\ \hline
4.106& $(0,2)$ &0 &$1/q+q$\\ \hline
4.107& $(-2,2)$  &0&$1/q+q$ \\ \hline
{\bf 4.108}& $(0,0)$  &0 &$ 1/q^5t^2 + 1/qt+1/q + q + qt + q^5t^2$\\ \hline
5.2426 & $(2,4)$  &$-4$ &$ 1/q^{11} t^3+1/q^9 t^3+1/q^7 t^2+1/q^5 t^2 +1/q^5 + 1/q^3   $\\ \hline
5.2427 & $(0,4)$  & $-2$ &$1/q^{9}t^{3}+1/q^{5}t^{2}+1/q^{3}+1/q$\\ \hline
5.2428 & $(0,2)$  & 0 &$1/q^7 t^3+ 1/q^3 t^2 +2/q +  q + q^3 t$ \\ \hline
5.2429 & $(2,4)$  & $-2$ &$1/q^{11} t^4+1/q^9 t^3+1/q^7 t^3+1/q^5 t^2 +1/q^5t + 1/q^3+2/q$\\ \hline
5.2430 & $(0,2)$  & 0 &$1/q+q$\\ \hline
5.2431 & $(-2,2)$  & 0 &$1/q+q$\\ \hline
5.2432 & $(-2,2)$  & 0 &$1/q+q$\\ \hline
5.2433 & $(0,4)$  & $-4$&$ 1/q^{11} t^3+1/q^9 t^3+1/q^7 t^2+1/q^5 t^2 +1/q^5 + 1/q^3 $ \\ \hline
5.2434 & $(0,4)$  & $-2$ &$1/q^{9}t^{3}+1/q^{5}t^{2}+1/q^{3}+1/q$\\ \hline
5.2435 & $(0,2)$  & 0 &$1/q^7 t^3+ 1/q^3 t^2 +2/q +  q + q^3 t$\\ \hline
5.2436 & $(-2,2)$  & 0 &$1/q+q$\\ \hline
{\bf 5.2437} & $(2,2)$  & $-2$ &$1/q^{13}t^{5}+1/q^{9}t^{4}+ 1/q^{9} t^3+1/q^7 t^2+1/q^5 t^2 +1/q^3 t+1/q^3 + 1/q   $\\ \hline
5.2438 & $(0,2)$  & $-2$ &$1/q^{9}t^{3}+1/q^{5}t^{2}+1/q^{3}+1/q$\\ \hline
5.2439& $(0,2)$  & 0 &$1/q+q$\\ \hline
5.2440& $(0,2)$  & 0 &$1/q+q$\\ \hline
5.2441 & $(-2,2)$  & 0& $1/q+q$\\ \hline
5.2442 & $(-2,2)$  & 0 &$1/q+q$\\ \hline
5.2443& $(-2,0)$  & 0 &$ 1/q^5t^2 + 1/qt+1/q + q + qt + q^5t^2$\\ \hline
5.2444& $(-2,0)$  & 0 &$1/q+q$\\ \hline
{\bf 5.2445} & $(4,4)$ & $-4$ &$1/q^{15}t^{5}+1/q^{11}t^{4}+ 1/q^{11} t^3+1/q^7 t^2+1/q^5  +1/q^3    $\\ \hline
5.2446 & $(2,4)$ &  $-2$ &$1/q^{9}t^{3}+1/q^{5}t^{2}+1/q^{3}+1/q$\\ \hline
5.2447& $(0,2)$ & 0 &$1/q+q$\\ \hline
5.2448 & $(0,2)$ & 0 &$1/q+q$\\ \hline
{\bf 6.89187}& $(4,4)$ &$-4$ &$1/q^{17}t^{6}+1/q^{15}t^{5}+1/q^{13}t^{5}+1/q^{11}t^{4}+ 2/q^{11} t^3+2/q^7 t^2+1/q^5  +1/q^3    $\\ \hline
6.89188 & $(2,4)$ &$-2$ &$1/q^{9}t^{3}+1/q^{5}t^{2}+1/q^{3}+1/q$\\ \hline
6.89189 & $(0,2)$ & $-2$ &$1/q^{9}t^{3}+1/q^{5}t^{2}+1/q^{3}+1/q$\\ \hline
{\bf 6.89198}& $(0,0)$ & 0 &$ 1/q^7t^3 +1/q^3t^2 + 1/q^3t+2/q + 2q + q^3t + q^3t^2+ q^7t^3$\\ \hline
6.90101 & $(0,4)$ & 0 &$1/q+q$\\ \hline
6.90102 & $(0,4)$ & 0 &$1/q+q$\\ \hline
6.90103 & $(-2,2)$ & 0 &$1/q+q$\\ \hline
6.90104 & $(0,4)$ & 0& $1/q+q$\\ \hline
6.90105 & $(-2,2)$ &0 &$1/q+q$\\ \hline
6.90106 & $(-2,2)$ & 0 &$1/q+q$\\ \hline
6.90107 & $(0,4)$ & 0 &$1/q+q$\\ \hline
6.90108& $(-2,2)$ & 0 &$1/q+q$\\ \hline
6.90109 & $(2,4)$ & $-4$&$ 1/q^{11} t^3+1/q^9 t^3+1/q^7 t^2+ 1/q^5 t^2 +1/q^5 + 1/q^3 $ \\ \hline
6.90110 & $(2,4)$ & $-2$ & $1/q^{13}t^{5}+1/q^{11}t^{4}+1/q^{9}t^{4}+ 1/q^{9} t^3 +1/q^7 t^3+1/q^7 t^2+1/q^5 t^2 $ \\  
&&& \quad $+1/q^5 t +1/q^3 t+1/q^3 + 2/q$\\ \hline
6.90111 & $(0,4)$ & $-2$ &$1/q^{9}t^{3}+1/q^{5}t^{2}+1/q^{3}+1/q$\\ \hline
6.90112 & $(0,4)$ & $-2$ &$1/q^{9}t^{3}+1/q^{5}t^{2}+1/q^{3}+1/q$\\ \hline
6.90113 & $(0,4)$ & $-2$ &$1/q^{9}t^{3}+1/q^{5}t^{2}+1/q^{3}+1/q$\\ \hline
6.90114 & $(0,4)$ & $-2$ &$1/q^{9}t^{3}+1/q^{5}t^{2}+1/q^{3}+1/q$\\ \hline
6.90115 & $(0,2)$ &  $-2$ &$1/q^9 t^3+1/q^7 t^2+ 1/q^5 t^2+ 1/q^3 t +1/q^3+1/q + t/ q + q^3 t^2$\\ \hline
6.90116 & $(0,2)$ & $0$ &$1/q^7 t^3+1/q^3 t^2+  2/q + q + q^3 t$\\ \hline
6.90117 & $(0,4)$ & $-2$ &$1/q^{11}t^{4}+1/q^{9}t^{3}+1/q^{7}t^{3}+ 1/q^{5} t^2+1/q^5 t+1/q^3  +2/q    $\\ \hline
6.90118 & $(0,2)$ & $0$ &$1/q^7 t^2+1/q^3 t+  t/q +1/q+ q + q^3 t^2$\\ \hline
6.90119 & $(0,4)$ & 0 &$1/q+q$\\ \hline
6.90120 & $(-2,2)$& 0 &$1/q+q$\\ \hline
6.90121 & $(0,4)$ & 0 &$1/q+q$\\ \hline
6.90122 & $(-2,2)$ &0 &$1/q+q$\\ \hline
6.90123 & $(0,2)$ &  $0$& $1/q^7 t^3+1/q^5 t^2+1/q^3 t^2+1/q t+  2/q + q +qt+ q^3 t+q^5 t^2$\\ \hline

\end{tabular}
}
\end{tabular}
\bigskip
\end{table}

\newpage

\begin{table}[H] 
\tiny
 \begin{tabular}{cccc}

{\rowcolors{1}{white!80!white!50}{white!70!white!40}
\begin{tabular}{|l|c|c|l|}
\hline
\tiny

{Virtual} &  \tiny {Signatures}    & \tiny{Rasmussen} & \tiny{Khovanov}  \\ 
\ \tiny {Knot} &   \tiny    $(\sigma_\xi^*,\sigma_{\xi})$            & \tiny{Invariant}&\tiny{Polynomial}
\\ \hline \hline

6.90124 & $(-2,2)$ & $0$& $1/q^3 t+  1/q + 2q + q^3 t^2+q^7 t^3$\\ \hline
6.90125 & $(0,4)$ &  $-2$&$1/q^{9}t^{3}+1/q^{5}t^{2}+1/q^{3}+1/q$ \\ \hline
6.90126 & $(0,2)$ &  0&$1/q+q$ \\ \hline
6.90127 & $(-2,4)$ & 0 &$1/q+q$\\ \hline
6.90128 & $(-2,2)$ & 0 &$1/q+q$\\ \hline
6.90129 & $(-2,4)$ & 0 &$1/q+q$\\ \hline
6.90130 & $(-2,2)$ & 0 &$1/q+q$\\ \hline
6.90131 & $(-2,2)$ & $0$ &$1/q^5 t^2+  1/qt + 1/q+q + q t+q^5 t^2$\\ \hline
6.90132 & $(-2,2)$ & 0 &$1/q+q$\\ \hline
6.90133 & $(-2,2)$ & 0 &$1/q+q$\\ \hline
6.90134 & $(-2,2)$ & 0 &$1/q+q$\\ \hline
6.90135 & $(-4,2)$ & 0 &$1/q+q$\\ \hline
6.90136 & $(-2,2)$ & 0 &$1/q+q$\\ \hline
6.90137 & $(-4,2)$ & 0 &$1/q+q$\\ \hline
6.90138 & $(-4,0)$ & $2$&$q+q^{3}+q^{5}t^{2}+q^{9}t^{3}$ \\ \hline
6.90139 & $(2,4)$ & $-4$ &$1/q^{15}t^{6}+1/q^{11}t^{5}+1/q^{11}t^{3}+ 2/q^{9} t^3+1/q^7 t^2+2/q^5 t^2 +1/q^5 +1/q^3    $\\ \hline
6.90140 & $(0,4)$ & $-2$ &$1/q^{11}t^{4}+1/q^{9}t^{3}+ 1/q^{7} t^3+1/q^5 t^2+1/q^5 t +1/q^3 +2/q    $\\ \hline
6.90141 & $(-2,4)$ & 0&$1/q+q$ \\ \hline
6.90142 & $(-2,2)$ & $0$&$1/q^7 t^2+1/q^3 t+  1/q + q + t/q+q^3 t^2$ \\ \hline
6.90143 & $(-2,2)$ &0 &$1/q+q$\\ \hline
6.90144 & $(0,4)$ &  $0$&$1/q^{9}t^{4}+1/q^{5}t^{3}+1/q^{3}t+1/q+2q$ \\ \hline
6.90145 & $(-2,2)$ &  $0$&$1/q^{5}t^{2}+1/qt+1/q+q+qt+q^5t^2$ \\ \hline
6.90146 & $(-2,2)$ &  $0$ &$2/q^{5}t^{2}+2/qt+1/q+q+2qt+2q^5t^2$\\ \hline
6.90147& $(0,4)$ &  $-4$ &$1/q^{11}t^{3}+ 1/q^{9} t^3+1/q^7 t^2+1/q^5 t^2 +1/q^5 +1/q^3    $\\ \hline
6.90148& $(0,4)$ &  $-2$ &$1/q^{9}t^{3}+1/q^{5}t^{2}+1/q^{3}+1/q$\\ \hline
6.90149 & $(0,4)$ & $-2$ &$1/q^{9}t^{3}+1/q^{5}t^{2}+1/q^{3}+1/q$\\ \hline
6.90150 & $(0,2)$ &$-2$& $1/q^{9}t^{3}+1/q^{7}t^{2}+1/q^{5}t^{2}+1/q^{3}t+1/q^3+1/q+t/q+q^3t^2$\\ \hline
6.90151 & $(0,4)$ & $-2$&$1/q^{11}t^{4}+1/q^{9}t^{3}+ 1/q^{7} t^3+1/q^5 t^2+1/q^5 t +1/q^3 +2/q    $ \\ \hline
6.90152 & $(0,4)$ & 0 &$1/q+q$\\ \hline
6.90153 & $(0,4)$ & 0 &$1/q+q$\\ \hline
6.90154 & $(0,2)$ &  $0$&$1/q^7 t^3+1/q^5 t^2+1/q^3 t^2+1/q t+  2/q + q + qt+q^3t+ q^5t^2$ \\ \hline
6.90155 & $(0,2)$ & $-2$ &$1/q^{9}t^{3}+1/q^{5}t^{2}+1/q^{3}+1/q$\\ \hline
6.90156 & $(0,4)$ & 0 &$1/q+q$\\ \hline
6.90157 & $(-2,4)$ &  0&$1/q+q$ \\ \hline
6.90158 & $(-2,2)$ &  $0$ &$1/q^5 t^2+1/q t+  1/q + q + qt+ q^5t^2$\\ \hline
6.90159 & $(-2,2)$ &  0 &$1/q+q$\\ \hline
6.90160 & $(-2,2)$ &  0 &$1/q+q$\\ \hline
6.90161 & $(-2,2)$ & $0$ &$1/q+q$\\ \hline
6.90162 & $(-2,2)$ & 0 &$1/q+q$\\ \hline
6.90163 & $(-2,2)$ & 0 &$1/q+q$\\ \hline
6.90164 & $(-2,2)$ &$0$ &$1/q^3 t+  1/q + 2q + q^3t^2+ q^7t^3$\\ \hline
6.90165 & $(0,2)$ & $0$ &$1/q^{9}t^{4}+1/q^{5}t^{3}+1/q^{3}t+1/q+2q$\\ \hline
6.90166 & $(-2,2)$ & $0$ &$ 1/q^5t^2 + 1/qt+1/q + q + qt + q^5t^2$\\ \hline
6.90167 & $(2,4)$ &  $-4$ &$1/q^{15} t^5+1/q^{11} t^4+1/q^{11} t^3 +1/q^7 t^2+1/q^5 + 1/q^3   $\\ \hline
6.90168 & $(2,4)$ & $-2$ &$1/q^{13} t^5+1/q^{9} t^4+1/q^{9} t^3+1/q^{7} t^2+1/q^{5} t^2 +1/q^3 t+1/q^3 + 1/q   $\\ \hline
6.90169 & $(0,4)$ &  $-2$ &$1/q^{9}t^{3}+1/q^{5}t^{2}+1/q^{3}+1/q$\\ \hline
6.90170 & $(0,2)$ & $-2$ &$1/q^{9}t^{3}+1/q^{5}t^{2}+1/q^{3}+1/q$\\ \hline
6.90171 & $(0,2)$ & 0 &$1/q+q$\\ \hline
{\bf 6.90172} & $(0,0)$ & $0$ &$1/q^{7}t^{3}+1/q^{5}t^{2}+ 1/q^{3} t^2+1/q^3 t+1/qt +2/q +2q + qt+ q^3t$
\\   
&&& \quad $ + q^3t^2+ q^5t^2+ q^7t^3   $\\ \hline
6.90173& $(0,4)$ &  0 &$1/q+q$\\ \hline
6.90174 & $(-2,4)$ & 0 &$1/q+q$\\ \hline
6.90175 & $(-2,2)$ &  0 &$1/q+q$\\ \hline
6.90176 & $(-2,4)$ &  0 &$1/q+q$\\ \hline
6.90177 & $(-2,2)$ &  0 &$1/q+q$\\ \hline
6.90178 & $(-2,2)$ &  0 &$1/q+q$\\ \hline
6.90179 & $(0,4)$ &  0 &$1/q+q$\\ \hline
\end{tabular}
}
\end{tabular}
\bigskip
\end{table}

\newpage

\begin{table}[H] 
\tiny
 \begin{tabular}{cccc}

{\rowcolors{1}{white!80!white!50}{white!70!white!40}
\begin{tabular}{|l|c|c|l|}
\hline
\tiny

{Virtual} &  \tiny {Signatures}    & \tiny{Rasmussen} & \tiny{Khovanov}  \\ 
\ \tiny {Knot} &   \tiny    $(\sigma_\xi^*,\sigma_{\xi})$            & \tiny{Invariant}&\tiny{Polynomial}
\\ \hline \hline

6.90180 & $(-2,4)$ &  0 &$1/q+q$\\ \hline
6.90181 & $(-2,2)$ & 0 &$1/q+q$\\ \hline
6.90182 & $(-2,2)$ &  0 &$1/q+q$\\ \hline
6.90183 & $(-2,2)$ &  $0$&$ 1/q^5t^2 + 1/qt+1/q + q + qt + q^5t^2$ \\ \hline
6.90184 & $(-4,0)$ &  $2$ &$q+q^{3}+q^{5}t^{2}+q^{9}t^{3}$\\ \hline
6.90185 & $(0,4)$ &  $-4$ &$ 1/q^{11} t^3+1/q^9 t^3+1/q^7 t^2+1/q^5 t^2 +1/q^5 + 1/q^3 $\\ \hline
6.90186 & $(0,4)$ &  $-2$ &$1/q^{9}t^{3}+1/q^{5}t^{2}+1/q^{3}+1/q$\\ \hline
6.90187 & $(0,4)$ &  $-2$ &$ 1/q^{11} t^4+1/q^9 t^3+1/q^7 t^3+1/q^5 t^2 +1/q^5t + 1/q^3+2/q $\\ \hline
6.90188 & $(-2,4)$ & 0 &$1/q+q$\\ \hline
6.90189 & $(-2,2)$ &  0 &$1/q+q$\\ \hline
6.90190 & $(-2,2)$ &  0 &$1/q+q$\\ \hline
6.90191 & $(0,2)$ &  $0$ &$ 1/q^{7} t^2+1/q^3 t+ 1/q+q+t/q +1/q^3 t^2  $\\ \hline
6.90192 & $(0,4)$ & $0$ &$1/q^{9}t^{4}+1/q^{5}t^{3}+1/q^{3}t+1/q+2q$\\ \hline
6.90193 & $(-2,2)$ & $0$ &$1/q^{5}t^{2}+1/qt+1/q+q+qt+q^5t^2$\\ \hline
6.90194 & $(0,2)$ &  $0$ &$1/q+q$\\ \hline
6.90195 & $(2,4)$ & $-4$ &$1/q^{15} t^5+1/q^{11} t^4+1/q^{11} t^3 +1/q^7 t^2+1/q^5 + 1/q^3   $\\ \hline
6.90196 & $(0,4)$ & $-2$ &$1/q^{9}t^{3}+1/q^{5}t^{2}+1/q^{3}+1/q$\\ \hline
6.90197 & $(0,4)$ & $-2$ &$1/q^{9}t^{3}+1/q^{5}t^{2}+1/q^{3}+1/q$\\ \hline
6.90198 & $(-2,2)$ & 0 &$1/q+q$\\ \hline
6.90199 & $(-2,2)$ &  0 &$1/q+q$\\ \hline
6.90200 & $(-2,0)$ &  0 &$1/q+q$\\ \hline
6.90201 & $(2,4)$ & $-2$ &$1/q^{9}t^{3}+1/q^{5}t^{2}+1/q^{3}+1/q$\\ \hline
6.90202 & $(0,4)$ &0 &$1/q+q$\\ \hline
6.90203 & $(0,4)$ & 0 &$1/q+q$\\ \hline
6.90204 & $(-2,2)$ & 0 &$1/q+q$\\ \hline
6.90205 & $(0,4)$ &0 &$1/q+q$\\ \hline
6.90206 & $(-2,2)$ &  0&$1/q+q$ \\ \hline
6.90207 & $(-2,2)$ &  0 &$1/q+q$\\ \hline
6.90208 & $(-4,0)$ &  $2$&$q+q^{3}+q^{5}t^{2}+q^{9}t^{3}$ \\ \hline
{\bf 6.90209} & $(2,2)$ &  $-2$& $ 1/q^{11} t^4+1/q^9 t^3+1/q^7 t^3+1/q^7 t^2+1/q^5 t^2 +1/q^5t+1/q^3t + 1/q^3$\\  
&&& \quad $+2/q+t/q+q^3t^2$\\ \hline
6.90210 & $(0,2)$ &  $0$ &$ 1/q^5t^2 + 1/qt+1/q + q + qt + q^5t^2$\\ \hline
6.90211 & $(-2,0)$ & 0 &$1/q+q$\\ \hline
6.90212 & $(-2,0)$ &  0 &$1/q+q$\\ \hline
6.90213 & $(-4,-2)$ & $2$& $q+q^{3}+q^3t+q^{5}t^{2}+q^{7}t^{2}+q^{9}t^{3}+q^{9}t^{4}+q^{13}t^{5}$\\ \hline
6.90214 & $(0,2)$ &  $-2$ &$1/q^{13} t^5+1/q^{9} t^4+1/q^{9} t^3+1/q^{7} t^2+1/q^{5} t^2 +1/q^3 t+1/q^3 + 1/q   $\\ \hline
6.90215 & $(0,2)$ &  $-2$ &$1/q^{9}t^{3}+1/q^{5}t^{2}+1/q^{3}+1/q$\\ \hline
6.90216 & $(0,2)$ &  $-2$ &$1/q^{9}t^{3}+1/q^{5}t^{2}+1/q^{3}+1/q$\\ \hline
6.90217 & $(0,2)$ &  0 &$1/q+q$\\ \hline
6.90218 & $(0,2)$ &  0 &$1/q+q$\\ \hline
6.90219 & $(0,2)$ &  0 &$1/q+q$\\ \hline
6.90220 & $(0,2)$ &  $0$&$ 1/q^5t^2 + 1/qt+1/q + q + qt + q^5t^2$ \\ \hline
6.90221 & $(0,2)$ &  0 &$1/q+q$\\ \hline
6.90222 & $(-2,2)$ &  0 &$1/q+q$\\ \hline
6.90223 & $(-2,2)$ &  0 &$1/q+q$\\ \hline
6.90224 & $(-2,2)$ &  0 &$1/q+q$\\ \hline
6.90225 & $(-2,2)$ &  0 &$1/q+q$\\ \hline
6.90226 & $(-2,2)$ &  0 &$1/q+q$\\ \hline
{\bf 6.90227} & $(0,0)$ & $0$ &$ 1/q^{9} t^4+1/q^5 t^3+1/q^5 t^2+1/q^3 t+1/qt + 1/q+2q+qt+q^5t^2$\\ \hline
6.90228 & $(0,2)$ &  $-2$ &$2/q^{13} t^5+2/q^{9} t^4+1/q^{9} t^3+2/q^{7} t^2+1/q^{5} t^2 +2/q^3 t+1/q^3 + 1/q   $\\ \hline
6.90229 & $(0,2)$ &  $-2$ &$1/q^{9}t^{3}+1/q^{5}t^{2}+1/q^{3}+1/q$\\ \hline
6.90230 & $(-2,2)$ &  0 &$1/q+q$\\ \hline
6.90231 & $(-2,2)$ &  0 &$1/q+q$\\ \hline
6.90232 & $(0,2)$ &  $0$ &$1/q^{9}t^{4}+1/q^{5}t^{3}+1/q^{3}t+1/q+2q$\\ \hline
6.90233  & $(-2,2)$ &  $0$& $ 1/q^5t^2 + 1/qt+1/q + q + qt + q^5t^2$ \\ \hline
6.90234 & $(-2,2)$ &  0 &$1/q+q$\\ \hline
6.90235 & $(-2,2)$ & 0 &$1/q+q$\\ \hline
\end{tabular}
}
\end{tabular}
\bigskip
\caption{The signatures, Rasmussen invariants, and Khovanov polynomials of alternating virtual knots.}\label{ras-table}
\end{table}

\bibliographystyle{halpha}    

\end{document}